\documentclass{amsart}
\usepackage{amsmath}
\usepackage{amsfonts}
\usepackage{amsthm}
\usepackage{adjustbox}
\usepackage{mathrsfs}
\usepackage{amssymb}
\usepackage{eucal}
\usepackage{url}
\usepackage{cite}
\usepackage[all]{xy}
\usepackage{rotating}
\usepackage{graphicx}
\graphicspath{{./figures/}}
\usepackage{hyperref}
\usepackage[usenames,dvipsnames,svgnames,table]{xcolor}
\definecolor{badgerred}{rgb}{0.715,0.004,0.004}
\definecolor{burntorange}{rgb}{0.801,0.332,0.0}
\usepackage[font=small,labelfont={bf,sf}, textfont={sf},margin=0em]{caption}
\hypersetup{colorlinks,
  filecolor=black,
  linkcolor=MidnightBlue,
  citecolor=NavyBlue,
  urlcolor=RoyalBlue,
  bookmarksopen=true}
\hypersetup{bookmarksopenlevel=1}
\usepackage{epstopdf}
\usepackage{subfigure}
\usepackage{amscd}

\makeatletter
\newcounter{savesection}
\newcounter{apdxsection}
\renewcommand\appendix{\par
  \setcounter{savesection}{\value{section}}%
  \setcounter{section}{\value{apdxsection}}%
  \setcounter{subsection}{0}%
  \gdef\thesection{\@Alph\c@section}}
\newcommand\unappendix{\par
  \setcounter{apdxsection}{\value{section}}%
  \setcounter{section}{\value{savesection}}%
  \setcounter{subsection}{0}%
  \gdef\thesection{\@arabic\c@section}}
\makeatother

\numberwithin{equation}{section}

\makeatletter
\newcommand{\addresseshere}{%
  \enddoc@text\let\enddoc@text\relax
}
\makeatother

\newcommand{\mat}{\begin{pmatrix}}
\newcommand{\rix}{\end{pmatrix}}
\newcommand{\deter}{\left|\begin{matrix}}
\newcommand{\minant}{\end{matrix}\right|}

\newcommand{\R}{\mathbb{R}}

\newtheorem{thm}{Theorem}[section]
\newtheorem{cor}[thm]{Corollary}
\newtheorem{prop}[thm]{Proposition}
\newtheorem{lem}[thm]{Lemma}

\theoremstyle{definition}

\theoremstyle{remark}

\newtheorem{clm}[thm]{Claim}

\newcommand{\C}{\mathbb{C}}

\DeclareMathOperator{\divg}{div}
\DeclareMathOperator{\grph}{graph}
\DeclareMathOperator{\diam}{diam}

\title{The Structure of Translating Surfaces with Finite Total Curvature}

\author{Ilyas Khan}
\address{Ilyas Khan \\ University of Wisconsin--Madison \\480 Lincoln Drive\\ Madison, WI 53706, USA} \email{ikhan4@wisc.edu}

\begin{document}

\begin{abstract}
In this paper, we prove that any mean curvature flow translator $\Sigma^2 \subset \R^3$ with finite total curvature and one end must be a plane. We also prove that if the translator $\Sigma$ has multiple ends, they are asymptotic to a plane $\Pi$ containing the direction of translation and can be written as graphs over $\Pi$. Finally, we determine that the ends of $\Sigma$ are strongly asymptotic to $\Pi$ and obtain quantitative estimates for their asymptotic behavior.
\end{abstract}

\maketitle

\tableofcontents

\section{Introduction}

A proper, embedded hypersurface $\Sigma^2 \subset \R^3$ is a mean curvature flow translator if there is a unit vector $V$ such that $\Sigma$ satisfies the following equation:
\begin{equation}\label{transeq}
\vec{H} = V^\perp,
\end{equation}
where $\vec{H}$ is the mean curvature of $\Sigma$ at a point $p$, and $V^\perp$ is the component of $V$ perpendicular to the tangent plane to $\Sigma$ at $p$. This is equivalent to saying that the family of surfaces $\{\Sigma + tV \; :\; t \in (-\infty, \infty)\}$ satisfies the mean cuvature flow (MCF) equation,
\begin{equation}\label{mcfeqn}
(\partial_t \mathbf{x})^\perp = \vec{H}
\end{equation}
where $\mathbf{x}(p)$ is the vector in $\R^3$ corresponding to the position of the point $p \in \Sigma$. Translating solitons, or translators, are important singularity models for Type II singularities of the MCF and are in general highly significant examples of ancient, immortal and eternal solutions of MCF. 

In this paper, we prove the following structure theorems for embedded translators with finite total curvature. 

\begin{thm}\label{oneend} Let $\Sigma^2 \hookrightarrow \R^3$ be a complete embedded MCF translator with one end, and finite total curvature
\begin{equation}\label{hypothesis1}
\int_\Sigma |A|^2 < \infty.  
\end{equation}
Then, $\Sigma$ is a plane parallel to $V$, the direction of motion.
\end{thm}

We also prove

\begin{thm}\label{graphicalends} 
Let $\Sigma$ be a complete embedded translator with finite total curvature, as in Theorem \ref{oneend}. Outside of some ball $B_R$, the ends of $\Sigma$ may be written as the graphs of functions $u_i$ over a fixed plane $\Pi$. Furthermore, each $u_i$ decays at a rate $o\big(r^{-\frac{1}{4(1+\delta)}}\big)$ for any $\delta >0$ as $r \rightarrow \infty$. In particular, each $u_i$ decays radially at an exponential rate in any sector in $P$ that excludes the ray $-tV$, $t \in (r_0, \infty)$, where $r_0$ is any non-negative number. 
\end{thm}

These results can be seen as addressing translator versions of the analogous classical questions for minimal surfaces. We now briefly summarize the proof. 

The main tools used to prove Theorem \ref{oneend} are a strong maximum principle for translators, a weak maximum principle for the gradient of a graphical translator, and a lemma due to Leon Simon in \cite{S} which gives an approximate graphical decomposition of surfaces with small total curvature. We begin by cutting out a large ``high curvature" region of $\Sigma$, leaving only disjoint annular ends with small total curvature. We choose one of these ends, and paste in a disk with small total curvature bounded by the total curvature of the annular end. We then apply Simon's lemma to this new disk of small total curvature $\epsilon^2$, which is a translator outside of a small fixed radius. The lemma tells us that, away from some ``pimples" of small diameter, the surface is ``mostly" a graph with gradient bounded by $C\epsilon^{1/6}$ over some plane. Away from a neighborhood of the center, we use the strong maximum principle and Schoen's method of moving planes to show that the pimples must be graphs. Then, using the weak maximum principle for gradients, we show that these graphical pimples must also have small gradient bounded by $C\epsilon^{1/6}$. Thus, inside any annulus $B_\eta \setminus B_{C\eta\epsilon^{1/2}}$, any end can be written as a graph over a plane with gradient bounded by $C\epsilon^{1/6}$. If we fix an annulus $B_R \setminus B_{R^{-1}}$ we can take a blow-down subsequence $\lambda_i\Sigma$ that can be written in $B_R \setminus B_{R^{-1}}$ as a sequence of graphs over a single plane $\Pi$ with gradient bounded by $C\epsilon_i^{1/6}$ where $\epsilon_i \rightarrow 0$. Then, if we assume we have only one end, we can re-apply the moving planes method and the maximum principle to the blow-down sequence to extract Theorem \ref{oneend} as a consequence.

Note that the application of Simon's lemma requires that the density of $\Sigma$, the quantity $r^{-2}\mathcal{H}^2(B_r(x)\cap \Sigma)$, is uniformly bounded when taken over all centers $x$ in $\R^3$ and radii $r \in \R_{> 0}$. In Corollary \ref{uniformarearatios}, we prove the somewhat stronger result that this is true for all surfaces with finite total curvature.

In order to prove Theorem \ref{graphicalends}, we must prove that if we take two points in a convergent blow-down sequence, $\lambda_i > \lambda_{i+1}$, then the ``in between" annular region $\Sigma \cap (B_{R^{-1}\lambda^{-1}_{i+1}} \setminus B_{R\lambda^{-1}_i})$ can be written as a graph over the blow-down limit plane $\Pi$. This is shown by examining two cases: (1) that $\Sigma \cap (B_{R\lambda^{-1}_{i+1}} \setminus B_{R^{-1}\lambda_{i+1}^{-1}})$ has the same orientation as a graph over $\Pi$ as $\Sigma \cap (B_{R\lambda^{-1}_{i}} \setminus B_{R^{-1}\lambda_{i}^{-1}})$, and (2) that these two annular regions have different orientations as graphs over $\Pi$. In the first case, geometric considerations allow us to apply Schoen's method of moving planes and the strong maximum principle to obtain graphicality over $\Pi$. In the second case, we integrate the geodesic curvature over the boundary of $\Sigma \cap (B_{R^{-1}\lambda^{-1}_{i+1}} \setminus B_{R\lambda_i^{-1}})$ and use the Gauss-Bonnet theorem to show that the surface must have total curvature uniformly bounded below by a constant. This contradicts the fact that the end has arbitrarily small total curvature outside a large radius of our choice. Then, application of the maximum principle and moving along the blow-down sequence allows us to conclude that each end may be written as a graph with sublinear growth over a single plane.

In the final section, we prove the strong asymptotics of the ends $\Sigma_i = \textrm{graph}(u_i)$, $i=1, \ldots, M$. Our main tool is the maximum principle: first we use an exponential barrier to bound the growth of the functions $u_i$ in the upper half-plane as $x_1 \rightarrow \infty$. In particular, this decay implies the uniform boundedness of the $u_i$ on the entire domain. This in turn allows us to obtain a radially decaying barrier based on the modified Bessel function of the second type. This gives the stated asymptotics in Theorem \ref{graphicalends} and completes the proof.

\section{Preliminaries}

First, we prove a strong maximum principle for translators. 

\begin{lem}\label{lineardiff}
Let $V = (v_1, v_2,v_3)$ be a unit vector in the direction of motion of the translating graphs of $u_1$ and $u_2$, which are defined on an open set $\Omega \subset \R^2$. Then, $v = u_2-u_1$ satisfies an equation of the form 
\[\divg (a_{ij} D_jv) + b_{i} D_iv = 0
\]
where $a_{ij}$ has ellipticity constants $\Lambda$ and $\lambda$ depending only on the upper bounds for the gradients $|D u_i|$.
\end{lem}

\begin{proof}
The $u_\alpha$, $\alpha =1,2$, both satisfy the quasilinear elliptic equation
\begin{equation}\label{translatoreqn} \divg \bigg(\frac{Du_\alpha}{(1 +|Du_\alpha|^2)^{1/2}} \bigg) = \frac{-v_1D_1 u_\alpha - v_2 D_2 u_\alpha + v_3}{(1 +|Du_\alpha|^2)^{1/2}}.
\end{equation}
Set 
\[A(p) = \frac{p}{(1 +|p|^2)^{1/2}}\;\;\; B(p) = \frac{V\cdot (-p, 1) }{(1 +|p|^2)^{1/2}}.
\]
We follow the proofs of Gilbarg and Trudinger in \cite[Theorem 10.7]{GT} and Colding and Minicozzi in \cite[Lemma 1.26]{CM}. Let $v = u_2- u_1$, and $u_t = tu_2 + (1-t)u_1$. Let 
\[ a_{ij}(x) = \int_0^1 D_{p_j}A^i(Du_t) dt \;\; \text{ and }b_i (x) = \int_0^1 D_{p_i}B(Du_t) dt. 
\]
By the Fundamental Theorem of calculus and the chain rule, $v$ satisfies
\[ \divg(a_{ij}D_j v) + b_i D_i v = 0.
\]
All that remains to show is the uniform ellipticity of $a_{ij}$. We show this by demonstrating the positive definiteness of the matrix differential $DA$. If $\nu$ is a unit vector and $p \in \R^2$, then
\[ DA(p)\nu = \frac{\nu}{(1+ |p|^2)^{1/2}} - \frac{\langle p, \nu \rangle}{(1+ |p|^2)^{3/2}}p
\]
\begin{align*}(1+ |p|^2)^{3/2} \langle \nu, DA(p) \nu \rangle &= (1+ |p|^2) - \langle p, \nu \rangle^2 \\
&\ge (1+|p|^2) - |p|^2 = 1.
\end{align*}
Thus, $a_{ij} = \int_0^1 DA(u_t)dt $ is a positive definite matrix whose ellipticity constants depend on the upper bounds of $|Du_\alpha|$. 
\end{proof}

\begin{cor}\label{SMP}
Let $\Omega \subset \R^2$ be an open connected neighborhood of the origin. If $u_1, u_2 :\Omega \rightarrow \R$ satisfy the translator equation \eqref{translatoreqn} with respect to the direction $V = (v_1, v_2, v_3)$ with $u_1 \le u_2$ and $u_1(0) = u_2(0)$, then $u_1 \equiv u_2$.
\end{cor}

\begin{proof}
Immediate from Lemma \ref{lineardiff} and the strong maximum principle proved in \cite[Lemma 3.5]{GT}. 
\end{proof}

In addition to this maximum principle, we also obtain a weak maximum principle for the gradients of translators that are graphs over planes containing $V$, the direction of motion.

\begin{lem}\label{WMP}
If $V \in \R^2 \times \{0\} \subset \R^3$, $\Omega$ is a bounded domain in $\R^2 \times \{0\}$, and $u: \Omega \subset \R^2 \rightarrow \R$ satisfies \eqref{translatoreqn}, then the partial derivatives $D_1 u$ and $D_2 u$ achieve their maxima and minima on the boundary.
\end{lem}

\begin{proof}
We set $V = (1, 0, 0)$ without loss of generality. The equation \eqref{translatoreqn} becomes
\begin{equation}\label{partranslatoreqn}\divg \bigg(\frac{Du}{(1 +|Du|^2)^{1/2}} \bigg) = \frac{-D_1 u}{(1 +|Du|^2)^{1/2}}.
\end{equation}
We differentiate this equation with respect to $x_i$ to obtain
\[D_i \divg \bigg( \frac{Du}{(1 +|Du|^2)^{1/2}} \bigg) = \frac{-D_1 (D_i u)}{(1 +|Du|^2)^{1/2}} + \frac{D_1 u}{(1 +|Du|^2)^{3/2}} Du \cdot D (D_i u) 
\]
\[D_i \bigg(\frac{\Delta u}{(1+|Du|^2)^{1/2}} - \frac{D_j u D_k u D_{kj}u}{(1+ |Du|^2)^{3/2}} \bigg)= \frac{-(D_1 (D_i u))}{(1 +|Du|^2)^{1/2}} + \frac{D_1 u (Du \cdot D (D_i u))}{(1+ |Du|^2)^{3/2}}
\]
Let $v = D_i u$. We calculate 
\begin{multline*}(1+|Du|^2)\Delta v - D_j u D_k u D_{kj} v \\ 
= 2 D_j v D_k u D_{jk} u - (D_1 v)(1 +|Du|^2) + D_1 u (Du \cdot D v)\\
+ \Delta u (Dv \cdot Du) - \frac{3 D_j u D_k u D_{kj} u(Du \cdot Dv)}{1+|Du|^2}
\end{multline*} 
This implies that $v$ satisfies a differential equation of the form $Lv = 0$, where the differential operator $L = a_{jk}D_{jk} + b_j D_j$ has coefficients given by
\[a_{jk} = (1 + |Du|^2)\delta_{jk} - D_j u D_k u 
\]
\[b_j =(1+|Du|^2)\delta_{1j} - 2D_kuD_{jk}u - \bigg(D_1u + \Delta u  - \frac{3D_kuD_l u D_{kl} u}{1+|Du|^2}\bigg)D_ju 
\]
We calculate
\[a_{jk}\xi_j \xi_k = (1+|Du|^2)|\xi|^2 - (Du \cdot \xi)^2 \ge  (1+|Du|^2)|\xi|^2 - |Du|^2|\xi|^2 \ge |\xi|^2.
\]
Similarly
\[a_{jk}\xi_j \xi_k \le (1+2|Du|^2)|\xi|^2.
\]
Since $u$ is assumed to be smooth and $\Omega$ is bounded, the operator $L$ is uniformly elliptic. In particular, $L$ satisfies the hypotheses of \cite[Theorem 3.1]{GT} and by applying this theorem, we obtain the weak maximum/minimum principle for $v$.
\end{proof}

We now prove a result about the pointwise decay of the second fundamental form $|A|$ and its derivatives. To obtain this result, we use Ecker's $\epsilon$-Regularity theorem for the mean curvature flow of surfaces.

\begin{thm}\label{eckregthm}\cite[Ecker Regularity Theorem]{E1} There exist constants $\epsilon_0 > 0$ and $c_0 >0$ such that for any solution $(M_t)_{t \in [0,T)}$ of mean curvature flow, any $x_0 \in \R^3$, and $\rho \in (0, \sqrt{T})$ the inequality
\begin{equation}\label{energybound}
\sup_{[T-\rho^2, T)} \int_{M_t \cap B_\rho (x_0)} |A|^2 \le \epsilon_0
\end{equation}
implies the mean value estimate
\begin{equation}\label{meanvalest} 
\sup_{[T - (\rho/2)^2, T)} \sup_{M_t \cap B_{\rho/2}(x_0)} |A|^2 \le c_0 \rho^{-4} \int_{T- \rho^2}^T \int_{M_t \cap B_\rho(x_0)} |A|^2.
\end{equation}
\end{thm}

\begin{lem}\label{curvdecay}
For a translator $\Sigma$ with finite total curvature, if $\varrho>1$ is distance from the origin, then
\[ |A| = O(\varrho^{-1/2}).
\]
Furthermore, in the positive half-space defined by $V$, 
\[|A| = O(\varrho^{-1}).\]
\end{lem}
\begin{proof}
This is a consequence of Ecker's $\epsilon$-regularity theorem, Theorem \ref{eckregthm}. By the finite total curvature condition, there is a ball $B_R$ centered at the origin such that $\int_{\Sigma \setminus B_R} |A|^2 <\epsilon_0$. We pick any $x_0 \in \R^3 \setminus B_{2R}$, and take $2\rho^2 = \textrm{dist}_{\R^3}(x_0, B_{2R})$. Suppose that $x_0 \in \Sigma$ and take the ball $B_\rho(x_0)$. Consider the mean curvature flow solution $\Sigma_t = \Sigma + tV$ defined on the interval $t \in [-3\rho^2/4, \rho^2/4]$. Notice that $B_R \cap \Sigma + tV$ never intersects $B_\rho(x_0)$ for any time $t \in [-3\rho^2/4, \rho^2/4]$. Thus,
\[\sup_{[-3\rho^2/4, \rho^2/4 )} \int_{\Sigma_t \cap B_\rho (x_0)} |A|^2 \le \epsilon_0.
\]
Theorem \ref{eckregthm}, Ecker's regularity theorem, yields
\begin{align*}\sup_{[0, \rho^2/4)} \sup_{\Sigma_t \cap B_{\rho/2}(x_0)} |A|^2 & \le c_0 \rho^{-4} \int_{-3 (\rho/2)^2}^{(\rho/2)^2} \int_{\Sigma_t \cap B_\rho(x_0)} |A|^2 \\
& \le c_0 \rho^{-2} \epsilon_0
\end{align*}
Thus, $|A|^2(x_0) \le C\varrho^{-1}$, where $\varrho > 0$ is distance to the origin (which is comparable to $2\rho^2 = \textrm{dist}_{\R^3}(x_0, B_{2R})$). This proves the first statement of the Lemma.

Now we prove the improved curvature decay in the upper half-space $\mathbb{H} \subset \R^3$ of points in $\R^3$ with positive $V$ component. Suppose that $x_0 \in \mathbb{H} \cap \Sigma$ and that $\rho = \textrm{dist}_{\R^3}(x_0, B_{R})$. Consider the ball $B_\rho(x_0)$ and the MCF solution $\Sigma_t = \Sigma + tV$ defined on the interval $t \in [\rho^2, 0]$. Since $x_0\cdot V > 0$, $\Sigma_t \cap B_R$ never intersects $B_\rho(x_0)$. Thus, by Theorem \ref{eckregthm}, 
\[\sup_{[-\rho^2/4, 0]}\sup_{\Sigma_t \cap B_{\rho/2}(x_0)} |A|^2 \le c_0 \rho^{-2}\epsilon_0.
\]
Choosing $t=0$, we see that $|A|^2(x_0) \le C\rho^{-2}$. Since $\rho \approx \varrho$, where $\varrho > 0 $ is $|x_0|$, we conclude that $|A| = O(\varrho^{-1})$.
\end{proof}

\begin{cor}\label{normperpv}
If $\mathbf{n}(p)$ is the normal vector to a translator $\Sigma$ moving in the direction $V$ at the point $p$, then 
\[\mathbf{n}(p) \cdot V = O(\varrho^{-1/2}),
\]
where $\varrho$ is the distance between $p$ and the origin.
\end{cor}

\begin{proof}
\[|H|^2 \le 2|A|^2, \;\; H(p) = \mathbf{n}(p) \cdot V.  
\]
Therefore,
\[\mathbf{n}(p) \cdot V =  O(\varrho^{-1/2}).
\]
\end{proof}

Before we can begin the proof of the main theorems, we must show that the embedding of $\Sigma$ into $\R^3$ is proper and that the area ratios, centered at any point, are uniformly bounded. We state the following theorem of Hartman which bounds the intrinsic area growth.

\begin{thm}\label{hartman}\cite[Proposition 1.3]{LT} Let $\Sigma$ be a complete surface with finite total curvature. Then there exists a constant $C_1>0$ depending only on $\int_\Sigma |K| dA$, such that 
\[ \textrm{Area}(B^{\Sigma}_r(x)) \le C_1 r^2,
\]
for all $x \in \Sigma$ and for all $r > 0$, where $B^\Sigma_r(x)$ is the intrinsic ball of radius $r$.
\end{thm}

We will extend this statement to the analogous statement for extrinsic balls by comparing intrinsic and extrinsic distances. The following theorem due to Huber gives important information about the topology and geometry of $\Sigma$.

\begin{thm}[Huber] If $\Sigma$ is a complete surface of finite total curvature immersed in $\R^n$, then $\Sigma$ is conformally equivalent to a compact Riemann surface with finitely many points deleted. 
\end{thm}

In particular, we note that Huber's theorem implies that $\Sigma$ must have finitely many ends and bounded genus. By Huber's theorem, we have a bijective conformal map, $f: \Omega^* \hookrightarrow \Sigma$, from the annulus $\Omega^* := \C \setminus B_1(0)$ to a given end of $\Sigma$. Following \cite{MS}, we consider the family of rescalings
\[f_\epsilon(z) = \epsilon^{\alpha + 1}[f(z/\epsilon) - f(2/\epsilon)].
\]
According to \cite{MS}, there exists a sequence $\epsilon_k \rightarrow 0$ with the following properties:
\begin{itemize}
    \item[(i)] There exists a plane $L \in \textrm{Gr}(2,3)$ such that the Gauss map converges to the constant $L$ in $W^{1,2}_{loc}$ on the punctured complex plane.
    \item[(ii)] The maps $f_{\epsilon_k}$ converge uniformly on compact subsets of $\C \setminus \{0\}$ to a conformal mapping $f_0: \C \setminus \{0\} \rightarrow L \subset \R^3$ which satisfies $|\partial_{x_i}f_0(z)| = C|z|^\alpha$. 
\end{itemize}
If we consider $f_0$ as a holomorphic function $f_0: \C \setminus \{0\} \rightarrow \C$, it can be seen that $f_0$ is of the form
\[f_0(z) = Cz^{\alpha + 1} + D.
\]
However, if $\alpha > 0$, the graph of the map $f$ must have an self-intersection, since the $f_{\epsilon_k}$ are converging uniformly as maps into $\R^3$. This contradicts the assumption of embeddedness, and thus $\alpha$ must be zero. M\"{u}ller and \v{S}ver\'{a}k define the multiplicity of an end by the integer $m = \alpha + 1$. Therefore, in our situation each end of $\Sigma$ has multiplicity one.

Now we state a result of M\"{u}ller and \v{S}ver\'{a}k that allows us to relate extrinsic distances to distances in the parameter space $\Omega^*$. Given that each end has unit multiplicity, we simplify the statement of the theorem.

\begin{prop}\cite[Corollary 4.2.10]{MS}\label{distcomp} Given an end of $\Sigma$, consider the conformal parametrization $f: \Omega^* \rightarrow \Sigma \hookrightarrow \R^3$ given by Huber's theorem. For each $\epsilon >0$ there exists $R>0$ such that the following statement holds. If $Q \subset \Omega^*$ is a square such that $Q \cap \{z \in \C \;:\; |z| \le R\} = \emptyset$ and if $\xi_1, \xi_2 \in \Omega^*$ are neighboring vertices of $Q$, then 
\[e^\lambda(1-\epsilon)|\xi_1 - \xi_2| \le |f(\xi_1) - f(\xi_2)| \le e^\lambda(1+\epsilon)|\xi_1 - \xi_2|,
\]
where $\lambda$ is a constant.
\end{prop}

By \cite[Theorem 4.2.1]{MS}, for $\xi_1, \xi_2 \in \Omega^*$ as in Proposition \ref{distcomp}, $d_\Sigma(f(\xi_1), f(\xi_2))$ approaches $e^\lambda|\xi_1 - \xi_2|$ uniformly as $R \rightarrow \infty$. Taking piecewise-$C^1$-paths, we can now compare intrinsic and extrinsic distances on $\Sigma$. Thus, if $x_1, x_2 \in \Sigma \subset \R^3$, there exists a constant $C$ such that 
\[C^{-1} d_{\Sigma}(x_1, x_2) \le |x_1 - x_2| \le C d_{\Sigma}(x_1, x_2). 
\]
In particular, if $x_0 \in \Sigma$, $B^{\Sigma}_{Cr}(x_0)$ contains the connected component of $B_r(x_0)$ containing $x_0$. 

We must control the number of connected components in order to use this comparison to bound area ratios. We begin by showing that area ratios are uniformly bounded with respect to centers in any compact set.

\begin{prop}\label{arearatios}
If $\Sigma$ is a complete, embedded surface with finite total curvature and $K$ is a compact domain in $\R^3$, there exists a constant $\beta >0$ that bounds the following area ratios:
\begin{equation}\label{areabds}
\sup_{x \in K} \sup_{R>0} \frac{\mathcal{H}^2(\Sigma \cap B_R(x))}{R^2} < \beta.
\end{equation}
\end{prop}

\begin{proof}
M\"{u}ller and \v{S}ver\'{a}k prove in \cite[Corollary 4.2.5]{MS} that for a fixed point $x_0 \in \Sigma$, the intrinsic and extrinsic distance functions coincide in the limit. That is,
\begin{equation}\label{intrinsicextrinsic} \lim_{\textrm{dist}_\Sigma(x, x_0)\rightarrow \infty} \frac{\textrm{dist}_\Sigma(x, x_0)}{|x-x_0|} = 1.
\end{equation}
Thus, for very large radii $R$, the intrinsic balls are comparable to extrinsic balls in the following sense:
\[B^\Sigma_{R}(x_0) \subset B_R \cap \Sigma \subset B^\Sigma_{2R}(x_0).
\]
Let $A(r)$ denote the area of the geodesic ball $B_r^\Sigma(x_0)$. A result from \cite{Sh} cited in the proof of \cite[Corollary 4.2.5]{MS} tells us that
\[\lim_{r \rightarrow \infty} \frac{A(r)}{r^2} = 2\pi \chi_\Sigma - \int_\Sigma K = \textrm{total number of ends}.
\]
By the finite total curvature property and Huber's theorem, the right hand side is bounded by a fixed constant. Since intrinsic and extrinsic balls are comparable, there exists some 
\[\beta > 2\pi \chi_\Sigma - \int_\Sigma K\]
such that 
\[\frac{\mathcal{H}^2(\Sigma \cap B_R(x_0))}{R^2} < \beta.
\]
By comparison in the limit as $R\rightarrow \infty$, we can see that this property holds for all $x_0 \in K$, and not only for points contained in $\Sigma$. 

We have a family of lower semi-continuous functions defined on $x\in K$
\[\bigg\{\frac{\mathcal{H}^2(\Sigma \cap B_R(x))}{R^2}\bigg\}_{R \in [0,\infty)}.
\]
Take the supremum of this family with respect to the index $R$. This is again a lower semi-continuous function with finite values, since the limiting area ratio for every center is a fixed finite number. We apply the extreme value theorem for lower semi-continuous functions to this new function on the compact set $K$. This proves the proposition.
\end{proof}

A neighborhood of each end can be parametrized as in Proposition \ref{distcomp}. By Huber's theorem, the complement of these neighborhoods is contained inside some compact set $K \subset \R^3$. By Proposition \ref{arearatios}, we know that the area ratios are bounded for balls centered in this region. We now consider points $x \in \R^3 \setminus K$. We know that if $x \in \Sigma \setminus K$ and $C_x(B_r(x) \cap \Sigma)$ is the connected component of $B_r(x) \cap \Sigma$ containing $x$, then there exists a uniform constant $C_2 > 0$ such that 
\[ C_x(B_r(x) \cap \Sigma) \subset B_{Cr}^\Sigma(x),
\]
and thus,
\[\frac{\mathcal{H}^2(C_x(B_r(x) \cap \Sigma))}{\pi r^2} \le C_2^2 \frac{\textrm{Area}(B_{C_2r}^\Sigma(x))}{\pi (C_2r)^2}.
\]
For a general point $x \in \R^3 \setminus K$, we consider the set $\{x_1, \ldots, x_d\}$, which is the nearest point projection to the $d$ ends of the surface $\Sigma$. Then, 
\[B_r(x) \cap \Sigma \subset \bigcup_{i=1}^d B^\Sigma_{C_2r}(x_i).
\]
The theorem of Hartman, Theorem \ref{hartman}, implies that there is a constant $C_3 > 0$ depending on $C_1$ and $C_2$ such that
\[\frac{\mathcal{H}^2(B_r(x) \cap \Sigma)}{\pi r^2} \le d C_3. 
\]
Combining the above bound with Proposition \ref{arearatios} proves a uniform bound on area ratios.
\begin{cor}\label{uniformarearatios}
There exists $D >0$ such that 
\[\sup_{x \in \R^3} \sup_{R>0} \frac{\mathcal{H}^2(\Sigma \cap B_R(x))}{R^2} < D.
\]
\end{cor}

\section{Blowing Down}

In this section, we prove that each blow-down sequence of $\Sigma$ has a subsequential $C^1_{loc}$-limit.

\begin{lem}\label{blowdownplane2}Let $\Sigma$ be a translating end with finite total curvature. For every sequence $\lambda_i \rightarrow 0$ there is a subsequence ${\lambda_i}_j \rightarrow 0$ such that the $C^1_{loc}$-limit of the rescalings ${\lambda_i}_{j}\Sigma$ in $\R^3 \setminus \{0\}$ is a plane $\Pi$ that is parallel to $V$. 
\end{lem}

The central tool that we will use to prove this lemma is the following graphical decomposition lemma due to L. Simon, \cite[Lemma 2.1]{S}. We restate the version of the lemma given in \cite{I} here for convenience. 

\begin{thm}[Simon's Lemma on $|A|^2$]\label{simonslemma2} For each $\beta > 0$, there is a constant $\epsilon_0 = \epsilon_0(\beta)$ such that if $\Sigma$ is a smooth 2-manifold properly embedded in $B_R \subset \R^3$ and 
\[\int_{\Sigma \cap B_R} |A|^2 \le \epsilon^2 \le \epsilon_0^2, \;\;\; \mathcal{H}^2(\Sigma \cap B_R) \le \beta R^2,
\]
then there are pairwise disjoint closed disks $\bar{P}_1, \ldots, \bar P_N$ in $\Sigma \cap B_R$ such that 
\begin{equation}\label{diambds2}\sum_m \diam (P_m) \le C(\beta)\epsilon^{1/2} R
\end{equation}
and for any $S \in [R/4, R/2]$ such that $\Sigma$ intersects $\partial B_S$ transversally and $\partial B_S \cap \bigcup_m \bar P_m = \emptyset$, we have
\[\Sigma \cap B_S = \cup_{i=1}^l D_i
\]
where each $D_i$ is an embedded disk. Furthermore, for each $D_i$, there is a 2-plane $L_i \subset \R^3$, a simply-connected domain $\Omega_i \subset L_i$, disjoint closed balls $\bar B_{i,p} \subset \Omega_i$, $p=1,\ldots, p_i$ and a function 
\[ u_i : \Omega_i \setminus \cup \bar B_{i,p} \rightarrow \R^3
\]
such that 
\begin{equation}\label{lipschitzbds2} \sup \bigg| \frac{u_i}{R} \bigg| + |Du_i| \le C(\beta) \epsilon^{1/6}
\end{equation}
and 
\[D_i \setminus \cup_m \bar P_m = \grph(u_i|\Omega_i \setminus  \cup \bar B_{i,p}).
\]
\end{thm}

Roughly speaking, this lemma tells us that a surface of small total curvature contained inside a ball can be expressed as the union of disks that are graphical away from a small pathological set of discs $\{P_i\}$, which we call ``pimples." We recall that Simon states on p. 289 of \cite{S} that each the boundary of each disk $D_i$ is a graphical curve contained in $\partial B_S \cap \Sigma$. In particular, this tells us a) that the $D_i$ can be taken to be disjoint (they cannot be connected by a pimple $P_m$, as these are all topological disks) and b) that if we assume $0 \in \Sigma \subset \R^3$, then the connected component of $\Sigma \cap B_S(0)$ containing $0$ is contained in a neighborhood of the cone
\[X_\epsilon(\Pi, 0) = \{y \in \R^3 : \textrm{dist}(y, \Pi) \le C\epsilon^\frac{1}{6}|\pi(y)|\}.
\]

\begin{lem}\label{graphicalradius}
Let $\epsilon,R>0$ be given, and let $\Sigma$ be a complete embedded translator, and select a point $x \in \Sigma$ such that
\[\int_{\Sigma \cap B_R} |A|^2 \le \epsilon^2 \le \epsilon_0^2, \;\;\; \mathcal{H}^2(\Sigma \cap B_R(x)) \le \beta\pi R^2, \text{ and }|x| > C\epsilon^{-1/3}.
\]
Let $\Sigma_x$ be the connected component of $\Sigma \cap B_S(x)$ containing $x$ (where $S$ is as in Theorem \ref{simonslemma2}) and consider the collection of pimples $\bar{P}_m \subset \Sigma$. Then there exists $S' \in [R/8, R/16]$ such that $\Sigma \cap B_{S'}$ is the graph of a function with gradient bounded by $C\epsilon^\frac{1}{6}$, defined on a plane $\Pi$ parallel to the direction of motion, $V$.
\end{lem}

\begin{proof}
Away from the pimples $\{P_m\}$, the component $\Sigma_x$ is the graph of a function with small gradient over a plane, which we denote $L_x$. By Corollary \ref{normperpv}, $\mathbf{n}(x) \cdot V = O(|x|^{-1/2})$. Since 
\[|x|^{-1/2} < C^{-1}\epsilon^{1/6},
\]
we can take $L_x$ to be a plane $\Pi$ that is parallel to $V$. Let $\pi: \R^3 \rightarrow \Pi$ be the standard orthogonal projection map. Now, consider the following cone:
\[X_\epsilon(\Pi, x) = \{y \in \R^3 : \textrm{dist}(y-x, \Pi) \le C\epsilon^\frac{1}{6}|\pi(y-x)|\}.
\]
By the graphical decomposition lemma, $\Sigma_x$ is contained in a neighborhood of $X_\epsilon(\Pi,x)$ with radius $C\epsilon^{1/2}R$:
\[\mathcal{N}_{C\epsilon^{1/2}R}(X_\epsilon(\Pi, x)) = \{ y + t\nu : y \in X_\epsilon(\Pi,x), \; t \in [0,C\epsilon^{1/2}R), \; \nu \in S^2 \}.
\]
Let $D_s$ denote the disk centered at $x$ of radius $s$ in the plane $\Pi$. This boundedness of $\Sigma_x$ ensures that for sufficiently small $\epsilon >0$, the intersections $\pi^{-1}(\partial D_s) \cap \Sigma_x$ are compact and do not intersect the boundary $\partial \Sigma_x$ for all $s \in [R/4,R/8]$. By the diameter bounds on the pimples and the transversality theorem, there is a number $S' \in [R/4,R/8]$ such that $\pi^{-1}(\partial D_{S'}) \cap \{P_m\} = \emptyset$ and $\pi^{-1}(\partial D_{S'})$ intersects $\Sigma_x$ transversely. By Theorem \ref{simonslemma2}, this intersection is a single graphical curve $\eta$ over $\partial D_{S'}$ that bounds a disk, $\pi^{-1}(D_{S'}) \cap \Sigma$.
\\
\\We now show that the surface given by $\pi^{-1}(D_{S'}) \cap \Sigma$ is a graph over $\Pi$. In order to prove this, we adapt the Alexandrov moving plane method of Schoen for minimal surfaces (\cite[Theorem 1.29]{CM} and \cite[Theorem 1]{Sc}). This has been previously done for translators in \cite{MSHS}.

\begin{thm}[Method of Moving Planes]\label{movplanes2} Let $\Omega \subset \Pi$ be an open set with $C^1$ boundary and let $\{\sigma_i\} \subset \R^3$ be simple closed curves each of which are graphs over distinct components of $\partial \Omega$ with bounded slope. Further assume that for any point $x \in \sigma_i \subset \Sigma$, the tangent plane to $\Sigma$ is well defined and does not contain the vertical normal vector $e_3$. Then any translator $\Sigma \subset \R^3$ contained in $\Omega \times \R$ with $\partial \Sigma = \cup_i \sigma_i$ must be graphical over $\Omega$.
\end{thm}

\begin{proof} Let $\Pi$ be spanned by unit vectors $e_1, e_2$ and let the normal direction be $e_3$. Given the plane $\{x_3 = t\}$, $\Sigma$ is divided into the portions $\Sigma^+_t$ above and $\Sigma^-_t$ below the plane. We reflect $\Sigma^+_t$ below the plane to obtain a new translator $\tilde \Sigma^+_t$ below the plane. Note that because $V$ is contained in $\Pi$, reflection in $\Pi$ preserves the translator equation. We decrease $t$ until we encounter the first point $p$ where (a)$T_p\Sigma$ contains the vector $e_3$ or (b) $\tilde \Sigma^+_t$ and $\Sigma^-_t$ have an interior point of contact. Let the critical height $t$ be called $t_0$. Suppose that (a) occurs--this means that the tangent planes $T_p \tilde \Sigma^+_{t_0}$ and $T_p \Sigma^-_{t_0}$ coincide and we may apply the maximum principle, Corollary \ref{SMP}, to show that $\Sigma$ must have a reflection symmetry through $\{x_3 = t_0\}$. This contradicts the graphicality of the boundary unless $\Sigma$ coincides with $\{x_3 = t_0\}$. 

If $\Sigma$ is not one-to-one (i.e. $t_0$ exists) and case (a) does not occur, then case (b) must occur. If there is a first interior point of contact, we may apply Corollary \ref{SMP} directly to show that $\Sigma$ has a reflection symmetry through $\{x_3 = t_0\}$. This is a contradiction, and thus neither case (a) nor case (b) occurs. Because $\partial \Sigma$ is a disjoint union of graphs with bounded slope over disjoint planar curves, there is no boundary point of contact. Thus, the projection of $\Sigma$ to $\Pi$ is one-to-one and $\Sigma$ is a graph.
\end{proof}

If we let $\Omega = D_{S'} \subset \Pi$ and $\sigma = \pi^{-1}(\partial D_{S'}) \cap \Sigma$, then $\pi^{-1}(D_{S'}) \cap \Sigma$ clearly satisfies the conditions of Theorem \ref{movplanes2} and is a graph over $D_{S'}$. Furthermore, we know from the graphical decomposition lemma, Theorem \ref{simonslemma2}, that $|Du|_\eta| \le C\epsilon^\frac{1}{6}$. By Lemma \ref{WMP}, the weak maximum principle for derivatives, $\pi^{-1}(D_{S'}) \cap \Sigma$ can be written as the graph of a function $u$ defined on $D_{S'} \subset \Pi$ with $|Du|\le C\epsilon^\frac{1}{6}$. This completes the proof of the lemma.
\end{proof}

Given some $\epsilon >0$, we can find $R_\epsilon >0$ such that 
\[\int_{\Sigma \setminus B_{R_\epsilon}} |A|^2 d\mathcal{H}^2 < \epsilon^2,
\]
and $\Sigma \setminus B_{R_\epsilon}$ decomposes into disjoint annular ends. Given one of these ends, $\Sigma_i$, we wish to glue a disk with small total curvature to the inner boundary. The resulting surface  $\Sigma_i'$ will be a topological disk without boundary to which we may apply Simon's Lemma (Theorem \ref{simonslemma2}). To this end, for each $R_\epsilon$, we find a cylindrical curve with special properties that will allow us to carry out the gluing construction. 

\begin{lem}\label{cylindricalcurve}
Given $\epsilon >0$ and $R_\epsilon > 0$ as above, there exists a plane $\Pi$ such that $V \in \Pi$, a radius $\rho>0$ such that $\pi(\Sigma \cap B_{R_\epsilon}) \subset D_\rho \subset \Pi$, and a graphical curve $\Gamma \subset \Sigma$ over $\partial D_{\rho} \subset \Pi$ such that
\[ \int_{\Gamma} |A|^2 d\mathcal{H}^1 \le \frac{C\epsilon^2}{\rho}.
\]
\end{lem}

\begin{proof}
By Corollary \ref{normperpv}, we may choose $R_\epsilon$ large enough that $\mathbf{n}\cdot V < C\epsilon^\frac{1}{6}$ in $\Sigma \setminus B_{R_\epsilon}$. Let $f(x) = |x|$. The norm of the tangential gradient $|\nabla^{\Sigma} f|$ is bounded by 1. The coarea formula says that 
\begin{align*} \int_{2R_\epsilon}^{3R_\epsilon}\int_{\Sigma \cap \{f = r\}} |A|d\mathcal{H}^1 dr &= \int_{\Sigma \cap \{2R_\epsilon < f <3R_\epsilon \}} |A||\nabla^\Sigma f| d\mathcal{H}^2
\end{align*}
Note that by H\"{o}lder's inequality and the assumed area bounds
\[\int_{\Sigma \cap \{2R_\epsilon < f <3R_\epsilon \}} |A|d\mathcal{H}^2 \le C R_\epsilon \bigg( \int_{\Sigma \cap \{2R_\epsilon < f <3R_\epsilon \}} |A|^2  d\mathcal{H}^2 \bigg)^{\frac{1}{2}}
\]
There exists some $R \in (2R_\epsilon, 3R_\epsilon )$ such that 
\begin{align}\label{coareasmallcurv} \int_{\Sigma \cap \{f = R\}} |A|d\mathcal{H}^1 & \le \frac{1}{R_\epsilon} \int_{\Sigma \cap \{2R_\epsilon < f < 3R_\epsilon \}} |A| d\mathcal{H}^2  \\ 
& \le C\bigg(\int_{\Sigma \cap \{2R_\epsilon < f < 3R_\epsilon \}} |A|^2 d\mathcal{H}^2\bigg)^{\frac{1}{2}} \nonumber \\
& \le C \epsilon. \nonumber
\end{align}

By embeddedness, we may assume that $\partial B_R$ intersects $\Sigma$ transversely and $\gamma = \partial B_{R} \cap \Sigma$ is a closed curve. From the inequality \eqref{coareasmallcurv}, the normal vector of $\Sigma$ has very small oscillation on the curve $\gamma$ for sufficiently small $\epsilon$. Thus, the curve $\gamma$ is contained in a small graphical annular region with gradient bounded by $C\epsilon^{1/6}$. The radius around the graphical neighborhood of each point in $\gamma$ is extended to $R/8$ by applying Lemma \ref{graphicalradius}. Approximately, this tells us that $\gamma$ is very close (on the order of $C\epsilon^{1/6}R$) to the cross section of some translate of $\Pi$ with the ball $B_\rho$. 

Ultimately, what we obtain from this argument is that the union of these neighborhoods of $\gamma$ contains the graph of a function $u$ defined on a planar annulus $\Omega = D_{R/16 + \alpha}\setminus D_{R/32 + \alpha}$ with width $R/32$ and $\alpha >0$. Note that we may choose $R>R_\epsilon$ large enough that $\pi(\Sigma \cap B_{R_\epsilon}) \subset D_{R/32} \subset \Pi$, where $\pi$ is orthogonal projection to $\Pi$.

Assume that $0 \in \Pi$ and let $g(x) = |\pi(x)|$ for all $x \in \R^3$. The coarea formula gives
\begin{align*}
    \int_{R/32+\alpha}^{R/16 + \alpha} \int_{\Sigma \cap \{ g = r\}} |A|^2 d\mathcal{H}^1 dr = \int_{\textrm{graph}(u)}|A|^2|\nabla^\Sigma g| d\mathcal{H}^2.
\end{align*}
Thus, there exists $\rho \in (R/32 + \alpha, R/16 + \alpha)$ such that
\[\int_{\Sigma \cap \{ g = \rho\}} |A|^2 d\mathcal{H}^1 \le \frac{32}{R} \int_{\textrm{graph}(u)}|A|^2 d\mathcal{H}^2,
\]
where we have used the fact that $|\nabla^{\Sigma} g| \le 1$. Furthermore, by construction, $\rho \le 17R/16$. Thus,
\begin{align}\label{coareacurve}
    \int_{\Sigma \cap \{ g = \rho\}} |A|^2 d\mathcal{H}^1 & \le \frac{C}{\rho} \int_{\textrm{graph}(u)}|A|^2 d\mathcal{H}^2\\
    &\le \frac{C\epsilon^2}{\rho}.
\end{align}
Note that $C>0$ is an absolute constant that does not depend on $R, R_\epsilon,$ or $\epsilon$. Letting the curve $\{g = \rho\}$ equal $\Gamma$ completes the proof of the lemma.
\end{proof}

In the proof of Lemma \ref{cylindricalcurve}, $\Gamma$ is a graph over $\partial D_\rho \subset \Pi$ and is contained in the graph of a function $u$, defined over an annulus $\Omega$ such that $\partial D_\rho \subset \Omega$. We now cut out the interior region in $\Sigma$ bounded by $\Gamma$, leaving only an annular end with total curvature bounded by $\epsilon^2$. We will appeal to the following lemma proved by L. Simon in \cite{S} to find a candidate disk with small total curvature  and boundary equal to $\Gamma$ to replace the excised region. 

\begin{lem}\cite[Lemma 2.2]{S}\label{biharmonicreplacement2}
Let $\Sigma^2 \subset \R^3$ be smooth embedded, $\xi \in \R^n$, $L$ a plane containing $\xi$, $u \in C^\infty(U)$ for some open $(L-)$neighborhood $U$ of $L \cap \partial B_\rho(\xi)$, and 
\[\textrm{graph } u \subset \Sigma, \;\;\;\; |Du|\le 1.
\]
Also, let $w \in C^\infty(L \cap \bar{B}_\rho(\xi))$ satisfy  
\begin{equation}\label{dirichletprob2}
  \left\{
     \begin{array}{lr}
       \Delta^2 w = 0 & \text{on } L \cap B_\rho(\xi)\\
       w= u, Dw = Du & \text{on } L \cap \partial B_\rho(\xi).
     \end{array}
   \right.
\end{equation} 
Then
\begin{equation}\label{hessline2}\int_{L\cap B_\rho(\xi)} |D^2w|^2 \le C\rho \int_{\Gamma} |A|^2 d\mathcal{H}^1,
\end{equation}
where $\Gamma = \textrm{graph}(u|L\cap \partial B_\rho(\xi))$, $A$ is the second fundamental form of $\Sigma$, and $\mathcal{H}^1$ is 1-dimensional Hausdorff measure (i.e. arc-length measure) on $\Gamma$; $C$ is a fixed constant independent of $\Sigma$, $\rho$. 
\end{lem}

Note that the solution $w$ exists and is unique \cite{S}, and that there is the following maximum principle for the biharmonic equation from \cite{PV}.

\begin{thm}[Maximum Principle for the Biharmonic Equation]\label{biharmmax2} 
If $\Delta^2u = 0$ in $D$, a bounded Lipschitz domain in $\R^n$, and $|D u|$ is continuous in $\bar{D}$, then there is a constant $C$ that depends only on the Lipschitz character of $D$ and independent of the diameter of $D$ such that 
\[||D u||_{L^\infty(D)} \le C||D u||_{L^\infty(\partial D)}.
\]
\end{thm}
Since $||Du||_{L^\infty(\partial B_\rho(\xi))} \le C\epsilon^{1/6}$, Theorem \ref{biharmmax2} implies that $||Du||_{L^\infty(B_\rho(\xi))} \le C\epsilon^{1/6}$, and thus there exists a uniform constant $C>0$ not depending on $R$ or $\epsilon$ such that
\[\int_{\textrm{graph }w} |A|^2 d\mathcal{H}^2 \le C \int_{L\cap B_\rho} |D^2w|^2.
\]
Combining this estimate with \eqref{hessline2} and Lemma \ref{cylindricalcurve}, we obtain the desired bound
\[\int_{\textrm{graph }w} |A|^2 d\mathcal{H}^2 \le C\epsilon^2,
\]
where $C$ is independent of $R$ and $\epsilon$.

We attach the graph of $w$ to the end of $\Sigma$ bounded by $\Gamma$. Simply joining these surfaces along $\Gamma$ results in a $C^1$ surface $S$, which does not satisfy the hypotheses of Theorem \ref{simonslemma2}: a $C^2$ surface is needed. To improve the regularity, we approximate the piecewise surface in the $H^2$-Sobolev norm by smooth functions. 

\begin{lem}\label{smoothsurgery}
Given $R>0$ as in Lemma \ref{cylindricalcurve} and a connected component of $\Sigma \setminus R_\epsilon$, we can find a smooth topological disk $\Sigma'$ such that $\Sigma' = \Sigma$ outside the ball $B_{2R}$.  
\end{lem}

\begin{proof}
Take a small tubular neighborhood $\mathcal{T}(\Gamma)$ of $\Gamma$ in $S$. On the outside of $\Gamma$, $\mathcal{T}(\Gamma)$ is equal to the graph of $u$ over the plane $\Pi$. On the inside of $\Gamma$, $\mathcal{T}(\Gamma)$ is equal to the graph of $w$ over $\Pi$. Let $v$ be the $C^1$ function over an annular domain $V \subset \Pi$ containing $\partial D_\rho$ such that $\textrm{graph }v = \mathcal{T}(\Gamma)$. 

First, we confirm that $v$ is in $W^{2,2}(V)$. Since $v \in C^1(V)$, we need only find weak second derivatives. Notice that $v$ is smooth away from $\partial D_\rho \subset V$. We define $D_{ij}v$ away from the measure zero set $\partial D_\rho$:
\begin{displaymath}
   D_{ij}v(x) = \left\{
     \begin{array}{lr}
       D_{ij}u & : x \in V \setminus \overline{B}_\rho \\
       D_{ij}w & : x \in V \cap B_\rho
     \end{array}
   \right.
\end{displaymath} 
Now, consider a test function $\varphi \in C_c^\infty(V)$. 
\begin{align*}
\int_{V} D_{ij}v \varphi dx &= \int_{V \setminus \overline{B}_\rho} D_{ij}u \varphi dx + \int_{V \cap B_\rho} D_{ij}w \varphi dx \\
&= - \int_{V \setminus \overline{B}_\rho} D_{i}u D_j\varphi dx - \int_{V \cap B_\rho} D_{i}w D_j\varphi dx + \int_{\partial B_\rho} (D_i u - D_i w) \nu_j dx \\
&= - \int_{V} D_{i}v D_j\varphi dx.
\end{align*}
Since $v \in C^1$, this is enough to show that the weak second derivatives we defined are valid and $v \in W^{2,2}$. Note that we used \eqref{dirichletprob2} to get rid of the last term in the second line.

We approximate $v$ in $W^{2,2}(V)$ by smooth functions. Let $\psi$ be a cutoff function that is uniformly equal to 1 in a neighborhood of $\partial D_\rho$ and compactly supported in $V$. Let $(\psi v)_h$ be the regularization of $\psi v$, with $h>0$ chosen so that $(\psi v)_h$ is also compactly supported in $V$ and 
\[||(\psi v)_h - \psi v||_{W^{2,2}(V)} < \epsilon^2. 
\]
Then the function 
\[ \tilde v = (1 - \psi) v + (\psi v)_h
\]
is equal to $v$ in a uniform neighborhood of the boundary $\partial V$ and has $||D^2 \tilde v||_{L^2(V)} < C\epsilon^2$. Since $|D\tilde{v}|$ is uniformly bounded, we have 
\[\int_{\textrm{graph }\tilde v} |A|^2 \le C \int_{V} |D^2 \tilde v|^2 < C\epsilon^2.
\]
Since $\textrm{graph }\tilde{v}$ matches $\Sigma$ in an open neighborhood of $\partial V$, the surface obtained by joining them, $\Sigma'$, is smooth and equal to $\Sigma$ outside $B_{2R}$.
\end{proof}

Choose $\epsilon> 0$ and let $Q > 0$ be a sufficiently large radius. We now prove the following lemma for $\Sigma'$. We assume that $0 \in \Sigma'$.

\begin{lem}\label{graphicalannuli}
Away from the ball $B_{C\epsilon^{1/2}Q}(0)$, the surface $\Sigma' \cap B_{Q/8}(0)$ is the graph of a function $u$ with gradient bounded by $C\epsilon^{1/6}$ over a plane $\Pi$ parallel to $V$. 
\end{lem}

\begin{proof}
We apply Simon's Lemma, Theorem \ref{simonslemma2}, to $\Sigma'$. Lemma \ref{graphicalradius} ensures that $\partial B_S \cap \Sigma$ (where $S \in [Q/4, Q/2]$) is an embedded, closed graphical curve bounding a connected topological disk. Away from pimples, Simon's Lemma tells us that $\Sigma' \cap B_S$ is the graph of a function $u$ defined on a domain in $\Pi$ with gradient bounded by $C\epsilon^{1/6}$. Let $\pi: \R^3 \rightarrow \Pi$ be orthogonal projection and let $R>0$ be as in Lemma \ref{smoothsurgery}. Let $D_r \subset \Pi$ be a disk such that $\pi(\Sigma' \cap B_{2R}) \subset D_r$. By the diameter bounds \eqref{diambds2}, there exists $s_1 \in [Q/4 - C\epsilon^{1/2}Q, Q/4]$ and $s_2 \in [r, r + C \epsilon^{1/2}Q]$ such that $\pi^{-1}(\partial B_{s_1})$ and $\pi^{-1}(\partial B_{s_2})$ do not intersect the pimples and intersect $\Sigma'\cap B_S$ transversely. We apply the Schoen/Alexandrov reflection method, Theorem \ref{movplanes2}, and the maximum principle, Lemma \ref{WMP}, as we did in Lemma \ref{graphicalradius} to show that the pimples lying between these cylinders are indeed graphical over $\Pi$ with gradient bounded by $C\epsilon^{1/6}$.  This completes the proof.
\end{proof}

We can now prove the main result of this section, Lemma \ref{blowdownplane2}.

\begin{proof}[Proof of Lemma \ref{blowdownplane2}]
Let us choose a sequence of annuli $B_{m} \setminus B_{1/m}$, where $m \in \mathbb{N}$, a sequence decreasing to zero $\epsilon_j \rightarrow 0$, and a family of rescalings $\lambda_i \Sigma'$, with $\lambda_i \rightarrow 0$. Let $\Sigma_j'$ be obtained for $\epsilon = \epsilon_j$ by cutting out a high curvature region and pasting in a disk as above. Note that in the inequality
\[\int_{\Sigma'_j} |A|^2 \le C\epsilon_j^2,
\]
the constant $C$ depends only on $\beta$, and is independent of $\epsilon_j$. Observe also that outside of some ball $\Sigma_j' = \Sigma_k'$, for any $j,k \in \mathbb{N}$.

Fix $m$ and $j$. By Lemma \ref{graphicalannuli}, for sufficiently large values of $i$, $\lambda_i \Sigma'$ can be written as a graph of a function over a plane $\Pi_i$ parallel to $V$ with $C^1$-norm bounded by $C\epsilon_j^{1/6}$ in $B_m \setminus B_{1/m}$. Un-fixing $j$, we see that by the compactness of the collection of planes $\{\Pi \in \textrm{Gr}_{2,3}: V \in \Pi \}$, there exists a subsequence $\lambda_{i_k}\rightarrow 0$ such that $\lambda_{i_k} \Sigma'$ converges to a fixed plane $\Pi$ inside $B_{m} \setminus B_{1/m}$ in the $C^1$-norm. 

Now, unfixing $m$, we take a further diagonal subsequence such that the rescalings of our designated end, $\lambda_{i_k} \Sigma$, converge to $\Pi \setminus \{0\}$ in the $C^1_{loc}$ topology. This concludes the proof of Lemma \ref{blowdownplane2}.
\end{proof}

Note that this blow-down limit is subsequential, and thus not necessarily unique. The problem arises from the fact that in Lemma \ref{graphicalannuli}, graphicality cannot be extended to a neighborhood of the origin with diameter proportional to the graphical radius. That is, the surface may be twisting and turning inside the ball $B_{C\epsilon^{1/6}Q}$ (in the notation of Lemma \ref{graphicalannuli}.) Thus, we must argue further in \S 4 to show that this asymptotic tangent plane is indeed unique. However, in the case that the translator $\Sigma$ has only one end, this possibility does not present a serious difficulty: Theorem \ref{oneend} is a quick corollary of Lemma \ref{blowdownplane2}.

\begin{proof}[Proof of Theorem \ref{oneend}]
Let $\{\lambda_i\Sigma\}$ be a blow-down sequence converging to $\Pi$. Fix $m \in \mathbb{N}$. Consider the sequence of annuli $\Sigma \cap B_{\lambda_{i}^{-1}m} \setminus B_{\lambda_{i}^{-1}m^{-1}}$. Each of these can be written as the graph of a function $u_i$ over $P$ with gradient $|Du_i| < C\epsilon_i^{1/6}$, where $\epsilon_i \rightarrow 0$. Consider the curve $\Sigma \cap \partial B_{\lambda_{i}^{-1}m}$ and its projection to an embedded Jordan curve in $\Pi$, which we denote $\sigma_i$. We take $\Omega$ to be the open set in $\Pi$ bounded by $\sigma_i$ and apply the method of moving planes, Theorem \ref{movplanes2} to determine that $\Sigma \cap B_{\lambda_{i}^{-1},}$ is a graph over $P$. Then we apply the weak maximum principle, Theorem \ref{WMP}, to determine that this graph has gradient $|Du| < C\epsilon_i^{1/6}$. As $m,i \rightarrow \infty$, the sequence $\Sigma \cap B_{\lambda_{i}^{-1}m}$ must be a sequence of graphical disks over $\Pi$ with increasingly small gradient. Thus, $\Sigma$ must be a plane.
\end{proof}

\section{Uniqueness of the Tangent Plane}

In this section, we establish the uniqueness of the asymptotic tangent plane and the graphicality of the ends of $\Sigma$. To prove this, we show that given two annuli in a convergent blow-down sequence, the ``interstitial space" between them must in fact be graphical with small gradient over the limit plane.

\begin{prop}\label{blowdownuniq}
Given an end of the translator $\Sigma$, the $C^1_{loc}$-limit in $\R^3 \setminus \{0\}$ of any blow-down sequence $\{\lambda_i \Sigma\}$, $\lambda_i \rightarrow 0$ is a unique plane $\Pi$ parallel to $V$. Consequently, the ends of $\Sigma$ can be written as graphs over $\Pi$ outside of some ball $B_{R_0} \subset \R^3$. 
\end{prop}

\begin{proof} Let $\epsilon>0$ and $\Sigma'$ be the modified surface from the previous section: a complete embedded topological disk, equal to the translator $\Sigma$ outside a ball $B_{R}$ and with total curvature less than $\epsilon^2$. Consider a blow-down sequence $\{\lambda_i \Sigma'\}$, with $\lambda_i \rightarrow 0$ that converges to a plane $\Pi$, and take $\lambda_j > \lambda_{j+1}$. Let $R>0$ be a large fixed radius, and let $j$ be large enough that $\Sigma' \cap B_{\lambda_j^{-1}R} \setminus B_{\lambda_j^{-1}R^{-1}}$ and $\Sigma' \cap B_{\lambda_{j+1}^{-1}R} \setminus B_{\lambda_{j+1}^{-1}R^{-1}}$ can be written as graphs of functions $u_j$ and $u_{j+1}$ over annuli in $\Pi$ with gradient bounded by $C\epsilon^{1/6}$. Let us also define the curves $\gamma_j$ and $\gamma_{j+1}$ which will denote the inner boundary $\Sigma' \cap \partial B_{\lambda_j^{-1}R^{-1}}$ and the outer boundary $\Sigma' \cap \partial B_{\lambda^{-1}_{j+1}R}$ respectively. 

If we project the space curves $\gamma_j$ and $\gamma_{j+1}$ to $\Pi$ via the orthogonal projection map $\pi: \R^3 \rightarrow \Pi$, we obtain simple closed planar curves $\sigma_j = \pi(\gamma_j)$ and $\sigma_{j+1} = \pi(\gamma_{j+1})$. Equivalently, $u_j(\sigma_j) = \gamma_j$ and $u_{j+1}(\sigma_{j+1})= \gamma_{j+1}$. Thus, $\sigma_j$ and $\sigma_{j+1}$ bound a large topological annulus $\Omega$ in $\Pi$. Let $\mathbf{n}$ be the normal vector to $\Pi$. Consider the solid cylinder $\Omega \times \mathbf{n}\R$. Suppose that $(\Sigma' \cap B_{\lambda_{j+1}^{-1}R}\setminus B_{\lambda_j^{-1}R^{-1}}) \cap (\sigma_j \times \mathbf{n}\R) = \gamma_{j}$. Then, $\Sigma' \cap (B_{\lambda_{j+1}^{-1}R} \setminus B_{\lambda_{j}^{-1}R^{-1}})$ is contained inside $\Omega \times \mathbf{n}\R$, and since the boundary curves $\gamma_{j}$ and $\gamma_{j+1}$ are graphs, we can apply the moving planes method, Theorem \ref{movplanes2}, to find that $\Sigma' \cap (B_{\lambda_{j+1}^{-1}R} \setminus B_{\lambda_{j}^{-1}R^{-1}})$ is a graph over $\Omega$.

Suppose that $(\Sigma' \cap B_{\lambda_{j+1}^{-1}R}) \cap (\sigma_j \times \mathbf{n}\R) \supsetneq \gamma_{j}$, i.e. the surface turns back and intersects the cylinder of small radius. Consider the domain $\mathcal{D} = \mathrm{int}(\sigma_j)$ in the plane $\Pi$ such that $\pi(\Sigma' \cap B_R) \subset \mathcal{D}$ and $D\mathcal{D}$ is bounded by the Jordan curve $\sigma_j$. We separate this situation into two cases. These cases are illustrated in Figure \ref{fig1} and Figure \ref{fig2} in Appendix \ref{figs}.

\textbf{Case 1:} $\Lambda := (\Sigma' \cap B_{\lambda_{j+1}^{-1}R} \setminus B_{\lambda_{j}^{-1}R^{-1}}) \cap (\mathcal{D} \times \mathbf{n}\R)$ is non-graphical or a multigraph over $D_R$ in either of the connected components of $(\mathcal{D} \times \mathbf{n}\R) \setminus B_{\lambda_{j}^{-1}R^{-1}}$. Without loss of generality, let us assume this occurs in the upper half space $\Pi \times \R_{\ge 0}$. In this case, we may use a modified method of moving planes. Consider the surface $\Sigma' \cap B_{\lambda_{j+1}^{-1} R}$ (a topological disk) contained in the cylinder $\textrm{int}(\sigma_{j+1}) \times \mathbf{n}\R$ with boundary equal to $\gamma_{j+1}$. If $(x_1,x_2)$ are coordinates in $\Pi$ and $x_3$ indicates height in $\mathbf{n}\R$, we take the plane $\{x_3= t\}$ parallel to $\Pi$ and the associated surfaces $\tilde {\Sigma'}_t^+$ and ${\Sigma'}_t^-$ as in the proof of Theorem \ref{movplanes2}. 

Let $h$ be the minimum height of $\Lambda$ in the upper half space $x_3 >0$. Note that
\[h > \frac{\lambda_{j}^{-1}R^{-1}}{2} > R,\]
where we may assume without loss of generality that $\lambda_j^{-1}$ is sufficiently large that the second inequality holds. The first inequality follows from repeated applications of Lemma \ref{graphicalannuli}, which tells us that once $\Sigma'$ leaves the ball $B_{\lambda_j^{-1}R^{-1}}$, it will never re-enter.

If a first point of interior tangency occurs, it must occur for $t > h$. This means that it occurs outside of $B_{R}$ and around this point of contact $\tilde {\Sigma'}_t^+$ and ${\Sigma'}_t^-$ must satisfy the translator equation. Thus, the strong maximum principle applies. This is a contradiction, so the first point of contact must occur on the boundary curve $\gamma_{j+1}$. However, $\gamma_{j+1}$ is a graph over $\Pi$, so there is no point of contact outside $\{x_3 = t\}$ on the boundary. Thus, the orthogonal projection $\pi(\Lambda)$ to $\Pi$ must be one-to-one, so Case 1 cannot occur.

\textbf{Case 2:} $\Lambda := (\Sigma' \cap B_{\lambda_{j+1}^{-1}R} \setminus B_{\lambda_{j}^{-1}R^{-1}}) \cap (\mathcal{D} \times \mathbf{n}\R)$ is a single-valued graph over $U$. Note that the plane $\Pi$ is spanned by two orthogonal vectors $\partial_{x_1} = V$ and $\partial_{x_2}$ and is normal to the coordinate vector $\partial_{x_3}$. Let us take the cross section of $\Sigma' \cap B_{\lambda_{j+1}^{-1}R} \setminus B_{\lambda_{j}^{-1}R^{-1}}$ with respect to the plane spanned by $\partial_{x_2}$ and $\partial_{x_3}$ passing through the origin. Since the normal part of $V$ has norm $|V^\perp| = O(\varrho^{-1/2})$ by Corollary \ref{normperpv}, the intersection with this plane is transverse for sufficiently large values of $\lambda_j^{-1}$. This intersection can be seen to have two connected components which can be distinguished by whether they intersect the inner boundary curve on the left or the right of the line $(x_2, x_3) \in \{0\} \times \R$. Without loss of generality, take the component that intersects $\gamma_j$ on the left. This curve is an embedded curve with boundary in $\gamma_j$ and $\gamma_{j+1}$, which we henceforth denote by $\eta$.

We now describe some properties of the curve $\eta$. Let $p_j$ and $p_{j+1}$ be the unique intersection points of $\eta$ with $\gamma_j$ and $\gamma_{j+1}$: these are the endpoints of $\eta$. Let $\ell$ be the length of $\eta$. We parametrize $\eta$ by arclength such that 
\[\eta(0) = p_j,\;  \eta(\ell) = p_{j+1},\text{ and }\langle \eta'(0), \partial_{x_2} \rangle = -1 + C\epsilon^{1/6}. \]
By assumption, $\eta$ crosses the line $\{0\} \times \R$ exactly once. We wish to use this condition to determine whether $\langle \eta'(\ell), \partial_{x_2} \rangle$ is close to 1 or -1. First note that 
\[\eta \cap B_{\lambda_{j+1}^{-1}R} \setminus  B_{\lambda_{j+1}^{-1}R^{-1}}
\]
is much longer than 
\[\eta \cap B_{\lambda_{j}^{-1}R^{-1}} \setminus  B_{\lambda_{j}^{-1}R^{-1}}.
\]
We then consider the two cases (1) $p_{j+1}$ is on the left of vertical axis $\{0\} \times \R$ and (2) $p_{j+1}$ is on the right of this axis. In case (1), $\eta$ must cross $\{0\} \times \R$ an even number of times, which cannot happen. Thus, (2) must hold. If $\langle \eta'(s), \partial_{x_2} \rangle \approx -1$ on $\eta \cap B_{\lambda_{j+1}^{-1}R} \setminus  B_{\lambda_{j+1}^{-1}R^{-1}}$, integration tells us that $p_{j+1}$ must be on the left of $\{0\} \times \R$, which is impossible. Thus, $\langle \eta'(\ell), \partial_{x_2} \rangle \approx 1$ in all situations. In particular, $\langle \eta'(\ell), \eta'(0) \rangle \approx -1$.

Let $\Sigma' \cap B_{\lambda_{j}^{-1}R} \setminus  B_{\lambda_{j}^{-1}R^{-1}}$ and $\Sigma' \cap B_{\lambda_{j+1}^{-1}R} \setminus  B_{\lambda_{j+1}^{-1}R^{-1}}$ be parametrized by 
\[X_j(x_1, x_2) = (x_1, x_2, u_j(x_1,x_2))\]
and
\[X_{j+1}(x_1, x_2) = (x_1, x_2, u_{j+1}(x_1,x_2))\]
respectively. The two surfaces are oriented by the respective 2-vectors
\[\partial_{x_1} X_{j} \wedge \partial_{x_2} X_{j} \text{ and } \partial_{x_1} X_{j+1} \wedge \partial_{x_2} X_{j+1}.
\]
Since these two surfaces are graphical annuli with very small gradient, their orientation 2-vectors are close to 
\[V \wedge \eta'(0) \text{ and } V \wedge \eta'(\ell),
\]
respectively. Note that at all points $\eta(s)$, $V$ is ``almost" a tangent vector to $\Sigma$ perpendicular to $\eta'(s)$, by Corollary \ref{normperpv}.  Since $\langle \eta'(\ell), \eta'(0) \rangle \approx -1$, we can deduce that
\[(\langle \partial_{x_1} X_{j} \wedge \partial_{x_2} X_{j})(p), (\partial_{x_1} X_{j+1} \wedge \partial_{x_2} X_{j+1})(q) \rangle < 0,
\]
for any $p \in \Sigma' \cap B_{\lambda_{j}^{-1}R} \setminus  B_{\lambda_{j}^{-1}R^{-1}}$ and $q \in \Sigma' \cap B_{\lambda_{j+1}^{-1}R} \setminus  B_{\lambda_{j+1}^{-1}R^{-1}}$. Equivalently, the normal vectors to the two annuli determine opposite orientations with respect to a fixed frame on $\R^3$. The crucial consequence of this fact is that the geodesic curvatures along curves ``nearby" $\gamma_j$ and $\gamma_{j+1}$ will have the same sign and be approximately equal to $2\pi$. We demonstrate this precisely in the remainder of the proof.

Next, we show that this change in orientation causes the total curvature to be large. By the same reasoning as in \eqref{coareacurve}, we can find curves $\Gamma_j$ and $\Gamma_{j+1}$ in $\Sigma' \cap (B_{\lambda_{j}^{-1}R} \setminus B_{\lambda_{j}^{-1}R^{-1}})$ and $\Sigma' \cap (B_{\lambda_{j+1}^{-1}R} \setminus B_{\lambda_{j+1}^{-1}R^{-1}})$ respectively such that
\[\pi(\Gamma_j) = \partial D_{\rho_j},\; \pi(\Gamma_{j+1}) = \partial D_{\rho_{j+1}} \text{ for some }\rho_j, \rho_{j+1} > 0,
\]
and such that
\[ \int_{\Gamma_j} |A|^2 d\mathcal{H}^1 \le \frac{2\epsilon^2}{\rho_j} \text{ and } \int_{\Gamma_{j+1}} |A|^2 d\mathcal{H}^1 \le \frac{2\epsilon^2}{\rho_{j+1}}.
\]

We consider the annular subset of $\Sigma' \cap (B_{\lambda_{j+1}^{-1}R} \setminus B_{\lambda_{j}^{-1}R^{-1}})$ bounded by the curves $\Gamma_j$ and $\Gamma_{j+1}$, which we call $S(\Gamma_j, \Gamma_{j+1})$. We apply the Gauss-Bonnet theorem:
\[\int_{S(\Gamma_j, \Gamma_{j+1})} K = 2\pi (2M_1 - 2g - M_0) - \int_{\partial S(\Gamma_j, \Gamma_{j+1})} \kappa_g,
\]
where $\kappa_g$ is the geodesic curvature, $M_1$ is equal to the number of connected components, $g$ is the genus, and $M_0$ is equal to the number of boundary components. Substituting $M_1 = 1,  M_0 = 2$ and $g = 0$, we obtain
\[\int_{\Sigma' \cap (B_{\lambda_{j+1}^{-1}R} \setminus B_{\lambda_{j}^{-1}R^{-1}})} K = - \int_{\Gamma_j} \kappa_g - \int_{\Gamma_{j+1}} \kappa_g.
\]
\begin{clm}\label{smallgeodcurv}
\[\int_{\Gamma_j}\kappa_g \;\;,\;\; \int_{\Gamma_{j+1}}\kappa_g \approx 2\pi
\]
\end{clm}
\begin{proof}
We apply similar reasoning to a proof on p.294-5 of \cite{S}. We parametrize $\Gamma_j$ and $\Gamma_{j+1}$ as follows.
\begin{equation} \label{gammaparam} \Gamma_j(\theta) = \rho_j e^{i\theta} + u_j(\rho_j e^{i\theta})\mathbf{n},
\end{equation}
and
\begin{equation} \label{gammaparam2} \Gamma_{j+1}(\theta) = \rho_{j+1} e^{i\theta} + u_j(\rho_{j+1} e^{i\theta})\mathbf{n},
\end{equation}
where $\theta \in [0,2\pi)$ and $\Pi$ is identified with $\C$. We calculate the integral
\begin{align*}\int_0^{2\pi} |\nabla^2 u(ie^{i\theta},ie^{i\theta})||\Gamma'(\theta)| d\theta &= \int_{\Gamma}|\nabla^2 u(ie^{i\theta},ie^{i\theta})| d\mathcal{H}^1 \\
& \le (\textrm{length}(\Gamma))^{1/2}\bigg(\int_\Gamma |\nabla^2 u|^2 d\mathcal{H}^1 \bigg)^{1/2} \\
& \le \rho^{1/2} \bigg( \frac{2\epsilon^2}{\rho} \bigg)^{1/2} \le 2\epsilon.
\end{align*}
Differentiating \eqref{gammaparam} and \eqref{gammaparam2}, and combining with the above calculation, we obtain that the total curvature vector $\vec{\kappa}$ can be writted in the following form:
\[\vec{\kappa} = \rho^{-1}\nu + E,
\]
where $E$ is a vector field on $\Gamma$ such that $\int_{\Gamma} |E| < C \epsilon^{1/6}$ and $\nu$ is the inward pointing unit normal to $\Gamma$ with respect to the surface $S(\Gamma_j, \Gamma_{j+1})$. Since
\[(\langle \partial_{x_1} X_{j} \wedge \partial_{x_2} X_{j})(p), (\partial_{x_1} X_{j+1} \wedge \partial_{x_2} X_{j+1})(q) \rangle < 0,
\]
for any $p \in \Sigma' \cap B_{\lambda_{j}^{-1}R} \setminus  B_{\lambda_{j}^{-1}R^{-1}}$ and $q \in \Sigma' \cap B_{\lambda_{j+1}^{-1}R} \setminus  B_{\lambda_{j+1}^{-1}R^{-1}}$, the geodesic curvatures $(\kappa_j)_g$ and $(\kappa_{j+1})_g$ have the same sign. We calculate
\begin{align*} \int_{\Gamma} \kappa_g &= \int_{\Gamma} \vec{\kappa} \cdot \nu = \int_0^{2\pi} (\rho + E\cdot \nu)|\rho d\theta\\
&= \int_0^{2\pi}\rho \rho^{-1} d\theta + \int_\Gamma E\cdot \nu\\
\end{align*}
\[\implies \bigg| \int_{\Gamma} \kappa_g - 2\pi \bigg| \le C\epsilon^{1/6}.
\]
This concludes the proof of the claim. 
\end{proof}

Claim \ref{smallgeodcurv} and the Gauss-Bonnet Theorem imply that 
\[\frac{3}{2}\int_{\Sigma' \cap (B_{\lambda_{j+1}^{-1}R} \setminus B_{\lambda_{j}^{-1}R^{-1}})} |A|^2 \ge 4\pi - \delta
\]
for some small $\delta >0$. Given the assumption of small total curvature, this is contradiction. Thus, Case 2 does not occur and $\Sigma' \cap (B_{\lambda_{j+1}^{-1}R} \setminus B_{\lambda_{j}^{-1}R^{-1}})$ must be the graph of a function $w_j$ over $\Omega \subset \Pi$. Since $|Dw_j| < C\epsilon^{1/6}$ on $\gamma_j$ and $\gamma_{j+1}$, the weak maximum principle implies that $|Dw_j|< C\epsilon^{1/6}$ on $\Omega$. 

Now, consider two blow down sequences $\{\mu_i \Sigma\}$ and $\{\lambda_j \Sigma\}$, with $\lambda_j, \mu_i \rightarrow 0$ and $\{\lambda_j \Sigma\}$ has blow-down limit equal to the plane $\Pi$. For any sufficiently large $\mu_i$, up to a subsequence of $\{ \lambda_j \}$, we can find $\lambda_j > \mu_i > \lambda_{j+1}$. The annular set $\Sigma' \cap (B_{\mu_i^{-1} R} \setminus B_{\mu_i^{-1} R^{-1}})$ is contained in the larger annulus $\Sigma' \cap (B_{\lambda_{j+1}^{-1}R} \setminus B_{\lambda_{j}^{-1}R^{-1}})$, which must be a graph over $\Pi$ with gradient bounded by $C\epsilon_j^{1/6}$. Now that $\Sigma' \cap (B_{\mu_i^{-1} R} \setminus B_{\mu_i^{-1} R^{-1}})$ can be written as a graph with gradient bounded by $C\epsilon_j^{1/6}$ over $\Pi$, we see that in fact the blow-down limit of $\{\mu_i \Sigma\}$ must be the plane $\Pi$ as well. 

In fact, this tells us that the entire end can be written as the graph of of a function $u$ over $\Pi$ with $|u| = o(\varrho)$ and $|Du| = o(1)$. That each end is a graph over the same plane $\Pi$ is a consequence of embeddedness.
\end{proof}

\section{Strong Asymptotics of the Ends}

In this section, we obtain upper bounds for the decay of an end of $\Sigma$ to its asymptotic plane. We begin by using a barrier argument to obtain unidirectional exponential decay of the ends. This is stated precisely in the following proposition:

\begin{prop}\label{UHPdecay}
Let $\Sigma_0$ be an end of a translator $\Sigma \subset \R^3$ that translates with unit speed in the $x_1$-direction and has finite total curvature. There exists a function $u: \R^2 \setminus B_R \rightarrow \R$ whose graph represents $\Sigma_0$ and whose magnitude decays uniformly at a rate $O(e^{-\alpha x_1})$ for any positive $\alpha < 1/2$ as $x_1 \rightarrow \infty$.  
\end{prop}

\begin{proof}
To prove this statement, we apply a barrier argument to the $x_1$-derivative of $u$, $D_{1}u$. We set $v:= D_{1}u$. By the proof of Lemma \ref{WMP}, $v = D_{1}u$ satisfies a differential equation of the form $Lv = 0$, where the differential operator $L = a_{jk}D_{jk} + b_j D_j$ has coefficients given by
\[a_{jk} = (1 + |Du|^2)\delta_{jk} - D_j u D_k u 
\]
\[b_j = (1+|Du|^2)\delta_{1j} - 2D_kuD_{jk}u  - \bigg(\Delta u + D_1u - \frac{3D_kuD_l u D_{kl} u}{1+|Du|^2}\bigg)D_ju. 
\]
We now calculate $L(v^2)$.
\begin{align*}
L(v^2) &= a_{ij}D^2_{ij}(v^2) + b_j D_j (v^2) \\
&= a_{ij}(D_i(2vD_jv)) + 2 v b_j D_j v \\
&= 2a_{ij}(D_iv D_jv) + 2va_{ij}D_{ij}^2v + 2vb_j D_j v \\
&= 2a_{ij}(D_iv D_jv)\\
&= 2(1+ |Du|^2)|Dv|^2 + 2D_iu D_j u (D_i v D_j v) \\
&= 2(1+ |Du|^2)|Dv|^2 + 2(\langle Du, Dv \rangle)^2.
\end{align*}
Thus,
\begin{align*}
L(v^2) &\ge 2(1+ |Du|^2)|Dv|^2\\
&\ge 0.
\end{align*}
We now choose our barrier function, $e^{-\frac{x_1}{2}}$. We calculate $L\big(e^{-\frac{x_1}{2}}\big)$:
\begin{align*}
L\big(e^{-\frac{x_1}{2}}\big) &= a_{11}D_{11}^2 \big(e^{-\frac{x_1}{2}}\big)  + b_1D_1\big(e^{-\frac{x_1}{2}}\big). 
\end{align*}
We calculate the coefficients $a_{11}$ and $b_1$ here:
\[a_{11}= (1+|Du|^2) - |D_1 u|^2 = 1 + |D_2 u|^2
\]
\[b_1 =  (1+|Du|^2) - |D_1u|^2 - 2D_ku D_{1k} u - \bigg( \Delta u - \frac{3D^2u(Du, Du)}{1+ |Du|^2} \bigg)D_1 u
\]
We now rely on the fact that $||Du||_{C^{k}} = o(1)$ as $\rho \rightarrow \infty$ in order to estimate the lower order terms for large values of $x_1$. 
\begin{align*}
L\big(e^{-\frac{x_1}{2}}\big) &= a_{11}\frac{e^{-\frac{x_1}{2}}}{4}  - b_1\frac{e^{-\frac{x_1}{2}}}{2} \\
&= (1 + |D_2 u|^2)\bigg(\frac{e^{-\frac{x_1}{2}}}{4}- \frac{e^{-\frac{x_1}{2}}}{2} \bigg)  \\
& \;\;\;\;\;\;\;\;\;\;\;\;\;\;\;\;\;\;  +\frac{e^{-\frac{x_1}{2}}}{2} \bigg( 2D_ku D_{1k} u + D_1u \Delta u - D_1 u\frac{3D^2u(Du, Du)}{1+ |Du|^2}\bigg) \\
&= -\frac{e^{-\frac{x_1}{2}}}{4} + e^{-\frac{x_1}{2}}o(1) \\
& \le -\frac{e^{-\frac{x_1}{2}}}{8} < 0,
\end{align*}
for $x_1$ sufficiently large. Let $M>0$ be such that the previous inequality holds for $x_1 \ge M$. Let $C$ be such that $Ce^\frac{-M}{2} > D_1u$ on the line $\{x_1 = M\}$. Note that such a $C$ must exist, since $|Du| = o(1)$ as the radius approaches infinity, and is in particular uniformly bounded. Now for some $\epsilon > 0$, define the function $\rho_\epsilon = Ce^{-\frac{x_1}{2}} + \epsilon$. Since there are no order zero terms, $L\rho_\epsilon < 0$ for $x_1 \ge M$. Thus,
\[L(\rho_\epsilon - v^2) < 0.
\]
Consider the semicircular domains $\mathcal{D}_S \subset \R^2$ parametrized by $S>0$ defined by the intersection of the balls $D_S(p)$, of radius $S$ and centered around $p = (x_1,x_2) = (M, 0)$, and the upper half-plane $\{x_1 > M\}$. Since $|Du| = o(1)$ as the radius approaches infinity, for each $\epsilon >0$, there exists some $S>0$ such that $\rho_\epsilon > v^2$ on the boundary $\partial \mathcal{D}_S$. The weak maximum principle \cite[Theorem 3.1]{GT} tells us that $\rho_\epsilon > v^2$ on $\mathcal {D}_S$. Sending $S \rightarrow \infty$ and then $\epsilon \rightarrow 0$, we see that $v^2 < C\epsilon^{-\frac{x_1}{2}}$ in the upper half-plane.

This in turn implies that $|D_1u| \le Ce^{-\frac{x_1}{4}}$ on $\{x_1 \ge M\}$. Integrating along rays in the direction of motion, we can conclude that $u$ approaches a limit along each ray parallel to the direction of motion. Fix real numbers $y$ and $\hat{y}$ and suppose that 
\[\lim_{x_1 \rightarrow \infty} u(x_1,y) = K, \; \; \lim_{x_1 \rightarrow \infty} u(x_1,\hat{y}) = \hat{K}.
\]
Since $|Du| = o(1)$, there exists a radius $R$ such that on $\R^2 \setminus B_R$, 
\[|Du| < \frac{|K-\hat{K}|}{2|y-\hat{y}|}.
\]
Thus, for sufficiently large values of $x_1$, $2|u(x_1,y) - u(x_1, \hat{y})| < |K-\hat{K}|$ only if $|K-\hat{K}| > 0$. This is a contradiction, so $K = \hat{K}$. Now, fix $x_2 = y$, for any $y \in \R$. The following formula holds:
\[\int_{x_1}^\infty D_1u(\xi, y) d \xi = K - u(x_1, y).
\]
By our bound $|D_1u| \le Ce^{-\frac{x_1}{4}}$,
\[u(x_1,y) \le K + \int_{x_1}^\infty Ce^{-\frac{\xi}{4}} d\xi \le K + Ce^{-\frac{x_1}{4}},
\]
and
\[u(x_1,y) \ge K - \int_{x_1}^\infty Ce^{-\frac{\xi}{4}} d\xi \ge K - Ce^{-\frac{x_1}{4}}.
\]
Since $y \in \R$ is arbitrary, this shows that $K \pm Ce^{-\frac{x_1}{4}}$ are upper and lower barriers for $u$, respectively. In particular, $u$ decays uniformly along rays parallel to $\nabla x_1$.
\end{proof}

Proposition \ref{UHPdecay} allows us to use grim hyperplanes and tilted grim hyperplanes as barriers for the ends. We obtain a number of immediate corollaries from this observation.

\begin{cor}\label{raydecay}
Along any ray in the coordinate plane $P$ not pointing in the direction $-V$, the magnitude of the graph $u$ decays exponentially.
\end{cor}

\begin{proof}
Immediate from comparison to tilted grim hyperplanes.
\end{proof}

\begin{cor}\label{uniformlybdd}
Each end is represented by the graph of a function $u$ that is uniformly bounded.
\end{cor}

\begin{proof}
This follows from comparison to grim hyperplanes above and below the top and bottom asymptotic planes, respectively.
\end{proof}

\begin{cor}
Any asymptotic plane of the translator $\Sigma$ must have distance less than $\pi$ from another asymptotic plane of $\Sigma$.
\end{cor}

\begin{proof}
If not, we could place a grim hyperplane between these two asymptotic planes, and obtain an interior point of contact with $\Sigma$.
\end{proof}

In order to improve on Proposition \ref{UHPdecay} and its consequences, we introduce another barrier, 
\[\phi = r^{1/4}e^{-x_1/2}K_0(r/2).\]
Here $K_0(x)$ denotes the modified Bessel function of the second kind, with parameter $0$, i.e. the solution to the modified Bessel's equation,
\[x^2 \frac{d^2y}{dx^2} + x\frac{dy}{dx} - x^2y = 0.
\]
The key point is that the decay and uniform boundedness of $u$ established in Proposition \ref{UHPdecay} allow us to find domains in which $\phi$ acts as a barrier for $u^2$.
\begin{prop}\label{finaldecay}
Let $\Sigma_0$ and $u: \R^2 \setminus B_R \rightarrow \R$ be as in Proposition \ref{UHPdecay}. 
\[|u| \le Cr^{1/\alpha}e^{-x_1/4}\sqrt{K_0(r/2)},
\]
where $\alpha > 4$. In particular, $u(x) = O\big(r^{-\frac{1}{4(1+\delta)}}\big)$, for $\delta > 0$ as $r \rightarrow \infty$. 
\end{prop}

\begin{proof} We first apply the linear operator $\mathcal{L} = \Delta + D_{1}$ to $\phi$.  
\begin{multline*}
    \Delta \phi + D_{1}\phi \\ = e^{-x_1/2}K_0(r/2)\mathcal{L}(r^{1/4}) + r^{1/4}\mathcal{L}\big(e^{-x_1/2}K_0(r/2)\big) + 2\langle \nabla e^{-x_1/2}K_0(r/2) , \nabla r^{1/4} \rangle.
\end{multline*}

We evaluate each summand separately.

\begin{align*}
    \mathcal{L}\big(e^{-x_1/2}K_0(r/2)\big) &= \frac{1}{4}e^{-x_1/2}K_0(r/2) + e^{-x_1/2}\Delta K_0(r/2) \\
    &\;\;\;\;\; - \langle e^{-x_1/2}\nabla x_1, \nabla K_0(r/2) \rangle + e^{-x_1/2} D_1 K_0(r/2) - \frac{1}{2}e^{-x_1/2}K_0(r/2)\\
    &= e^{-x_1/2}\Delta K_0(r/2) - \frac{e^{-x_1/2}}{4}K_0(r/2) \\
    &= e^{-x_1/2} \bigg(\frac{\partial^2}{\partial r^2} K_0(r/2) + \frac{1}{r} \frac{\partial}{\partial r} K_0(r/2) - \frac{1}{4}K_0(r/2) \bigg) \\
    &= e^{-x_1/2} \bigg( \frac{1}{4}\frac{\partial^2 K_0}{\partial r^2}(r/2) + \frac{1}{2r}\frac{\partial K_0}{\partial r} (r/2) - \frac{1}{4}K_0(r/2) \bigg) \\
    &= \frac{e^{-x_1/2}}{r^2}\bigg( \bigg(\frac{r}{2}\bigg)^2 \frac{\partial^2 K_0}{\partial r^2}(r/2) + \frac{r}{2}\frac{\partial K_0}{\partial r} (r/2) - \bigg(\frac{r}{2}\bigg)^2 K_0(r/2) \bigg) \\
    &= 0,
\end{align*}
by the definition of the modified Bessel function $K_0$. We now evaluate
\begin{align*}
    \mathcal{L}r^{1/4} &= \frac{\partial^2}{\partial r^2} r^{1/4} + \frac{1}{r}\frac{\partial}{\partial r}r^{1/4} + D_1 r^{1/4} \\
    &= -\frac{3}{16}r^{-7/4} + \frac{1}{4}r^{-7/4} + \frac{(e_1 \cdot e_r)}{4} r^{-3/4}\\
    &= \frac{1}{16}r^{-7/4} + \frac{(e_1 \cdot e_r)}{4} r^{-3/4}
\end{align*}
Now we calculate the final term. 
\begin{align*}
2\langle \nabla r^{1/4}, \nabla e^{-x_1/2}K_0(r/2) \rangle &= \frac{r^{-3/4}e^{-x_1/2}}{4}\frac{\partial K_0}{\partial r}(r/2) - \frac{r^{-3/4}e^{-x_1/2}}{4} K_0(r/2) (e_1 \cdot e_r) 
\end{align*}
Combining these results, we obtain
\begin{align*}
   \Delta \phi + D_{1}\phi &= 0 + e^{-x_1/2}K_0(r/2)\bigg(\frac{1}{16}r^{-7/4} + \frac{(e_1 \cdot e_r)}{4} r^{-3/4}\bigg)\\
   &\;\;\;\;\;\;\; + \frac{r^{-3/4}e^{-x_1/2}}{4}\frac{\partial K_0}{\partial r}(r/2) - \frac{r^{-3/4}e^{-x_1/2}}{4} K_0(r/2) (e_1 \cdot e_r) \\
   &= \frac{r^{-3/4}e^{-x_1/2}}{4}\frac{\partial K_0}{\partial r}(r/2) + o\bigg(\bigg|\frac{r^{-3/4}e^{-x_1/2}}{4}\frac{\partial K_0}{\partial r}(r/2)\bigg|\bigg).
\end{align*}
The asymptotic expansion at infinity of $K_0'(x)$ is as follows:
\[K_0'(x) \sim - \sqrt{\frac{\pi}{2x}}e^{-x}\bigg(1 + \frac{3}{8x} + O\bigg(\frac{1}{x^2}\bigg) \bigg).
\]
Thus, for sufficiently large $r$, $K_0'(r/2)$ is negative. Thus,
\[\mathcal{L}\phi < 0 \;\;\;\;\text{  for $r$ sufficiently large.}
\]
We now consider the operator $Q$ defined in \eqref{translatoreqn}, where we set $(v_1, v_2, v_3) = (1,0,0)$.
\begin{align*}Q\phi &= (1+|D\phi|^2)\mathcal{L}\phi - D^2\phi(D\phi, D\phi) \\ 
&= a_{ij} D^2_{ij}\phi + b_j D_j \phi
\end{align*}
where $(x, u, Du) = (x, z, p)$ and 
\[a_{ij} = 1 + |p|^2\delta_{ij} + p_ip_j
\]
\[b_{j} = (1+ |p|^2)\delta_{1j}
\]
We know that $\phi$ has the following asymptotic expansion
\[\phi(x_1,x_2) = r^{1/4}e^{-x_1/2}K_0(r/2) \sim Cr^{-1/4}e^{\frac{-x_1-r}{2}}\bigg(1 - \frac{1}{8r} + O\bigg(\frac{1}{r^2}\bigg) \bigg)
\]
Consequently, we know that for sufficiently large $r$, the partial derivatives $D_i\phi$ are bounded above by a multiple of $r^{-5/4}$. Thus, 
\[|D^2\phi(D\phi, D\phi)| = O(r^{-5/2}e^{-x_1/2}K_0(r/2)) = o(\mathcal{L}\phi)
\]
as $r \rightarrow \infty$. We conclude that $Q\phi < 0$. Consider the family of functions 
\[\phi_{\epsilon, N} = C_1\phi + C_2\exp(-(x_1 + N)/2) +\epsilon,
\]
where $C_1, C_2>0$ are fixed constants which depends on the upper bound of the function $u$ in the annulus $\R^2 \setminus B_{R_0}$. A quick calculation yields
\begin{multline*}
    Q(\exp(-(x_1 + N)/2)) = \frac{\exp(-(x_1 + N)/2)}{4} - \frac{\exp(-(x_1 + N)/2)}{2} \\ - \frac{\exp(-(x_1 + N))}{4}\bigg(\frac{\exp(-(x_1 + N)/2)}{4}\bigg).
\end{multline*}
This expression is clearly negative for $x_1 \ge -N$.  

Finally, we calculate $Q(u^2)$, where $u$ satisfies the translator equation \eqref{translatoreqn}, i.e. $Qu =0$. 
\begin{align*}
Q(u^2) &= a_{ij}D^2_{ij}(u^2) + b_j D_j (u^2) \\
&= a_{ij}(D_i(2uD_ju)) + 2 u b_j D_j u \\
&= 2a_{ij}(D_iu D_ju) + 2ua_{ij}D_{ij}^2u + 2ub_j D_j u \\
&= 2a_{ij}(D_iu D_ju)\\
&= 2(1+ |Du|^2)|Du|^2 + 2D_iu D_j u (D_i u D_j u) \\
&= 2(1+ |Du|^2)|Du|^2 + 2|Du|^4\\
&\ge 0.
\end{align*}
Now that we know that $Q(\phi_{\epsilon, N} - u^2) \le 0$, we may apply the following quasilinear maximum principle:

\begin{thm}\cite[Theorem 10.1]{GT} Let $u, v \in C^0(\bar{\Omega})\cap C^2(\Omega)$ satisfy $Qu \ge Qv$ in $\Omega$, $u \le v$ on $\partial \Omega$, where
\begin{itemize}
    \item [(i)] $Q$ is locally uniformly elliptic with respect to either $u$ or $v$;
    \item [(ii)]the coefficients $a_{ij}$ are independent of $z$;
    \item [(iii)] the coefficient $b(x,p) = b_j(x,p)p_j$ is non-increasing in $z$ for each $(x,p) \in \Omega \times \R^n$;
    \item [(iv)] the coefficients $a_{ij}, b$ are continuously differentiable with respect to the $p$ variables in $\Omega \times \R \times \R^n$.
\end{itemize}
It then follows that $u \le v$ in $\Omega$.
\end{thm}
It is clear that for sufficiently large $R_0$, conditions (i)-(iv) are satisfied by $u^2$ and $\phi$ on the annulus $\R^2 \setminus B_{R_0}$. Consider the large square $S_N$ with corners $(N,N), (N,-N), (-N,N), (-N,-N)$ and $N \ge R_0 > 0$. Let $\Omega_N = S_N \setminus B_{R_0}$. Set $C_1 > 0$ so that $C_1\phi|_{\partial B_{R_0}} \ge ||u^2||_{L^\infty(\R^2 \setminus B_{R_0})}$ and set $C_2 > ||u^2||_{L^\infty(\R^2 \setminus B_{R_0})}$. By Proposition \ref{UHPdecay}, Corollary \ref{raydecay}, and Corollary \ref{uniformlybdd}, $\phi_{\epsilon, N} > u^2$ on $\partial \Omega_N$ for sufficiently large $N$ and some arbitrary small $\epsilon >0$. On the edge from $(N, -N)$ to $(-N,-N)$, $x_1 \equiv -N$ and $\phi_{\epsilon, N} > C_2 > u^2$. The other three edges lie in a sector not containing $-V$, so by Corollary \ref{raydecay}, for sufficiently large values of $N$, $u^2 < \epsilon$ on these edges. Thus, given an $\epsilon$, for sufficiently large $N$, $\phi_{\epsilon, N} > u^2$ on $\partial \Omega_N$.

Having chosen $R_0 > 0$ sufficiently large, on the domains $\Omega_N$, we have $Q(\phi_{\epsilon, N}) < 0$, so we may apply \cite[Theorem 10.1]{GT} and obtain that $\phi_{\epsilon,N} - u^2 \ge 0$ on $\Omega_N$. Sending $N \rightarrow \infty$, we see that $C_1\phi + \epsilon \ge u^2$ on $\R^2 \setminus B_{R_0}$. Since $\epsilon >0$ is arbitrary, we send $\epsilon$ to zero and conclude that $C_1\phi \ge u^2$.

We finish the proof of the proposition by observing that all of the above arguments apply when we select $\phi = r^{2/\alpha} e^{-x_1/2} K_0(r/2)$ and $\alpha \in (4,\infty)$ is an arbitrary constant.
\end{proof}

\begin{proof}[Proof of Theorem \ref{graphicalends}]
The theorem immediately follows from Proposition \ref{blowdownuniq} and Proposition \ref{finaldecay}. 
\end{proof}

\section*{Acknowledgments}  
I would like to thank Bing Wang for suggesting this problem. I would also like to thank Bing Wang and Sigurd Angenent for many helpful conversations and their careful feedback on the drafts of this paper. I am also very grateful to Sigurd Angenent for his invaluable guidance in the writing of \S 5, especially for explaining the importance of the Bessel functions and sharing his overall wisdom around asymptotic formulae. Lastly, I'd like to thank Professor Yuxiang Li for pointing out that conditions on the area ratios could be dropped and suggesting the proof of Proposition \ref{arearatios}.

\addresseshere

\newpage
\appendix

\section{Illustrations for the Proof of Proposition 4.1}\label{figs}
\noindent\begin{minipage}{\textwidth}
    \centering
    \captionof{figure}{Case 1}
    \includegraphics[angle=90, origin=c, scale=0.25, trim={0 0 15cm 0},clip]{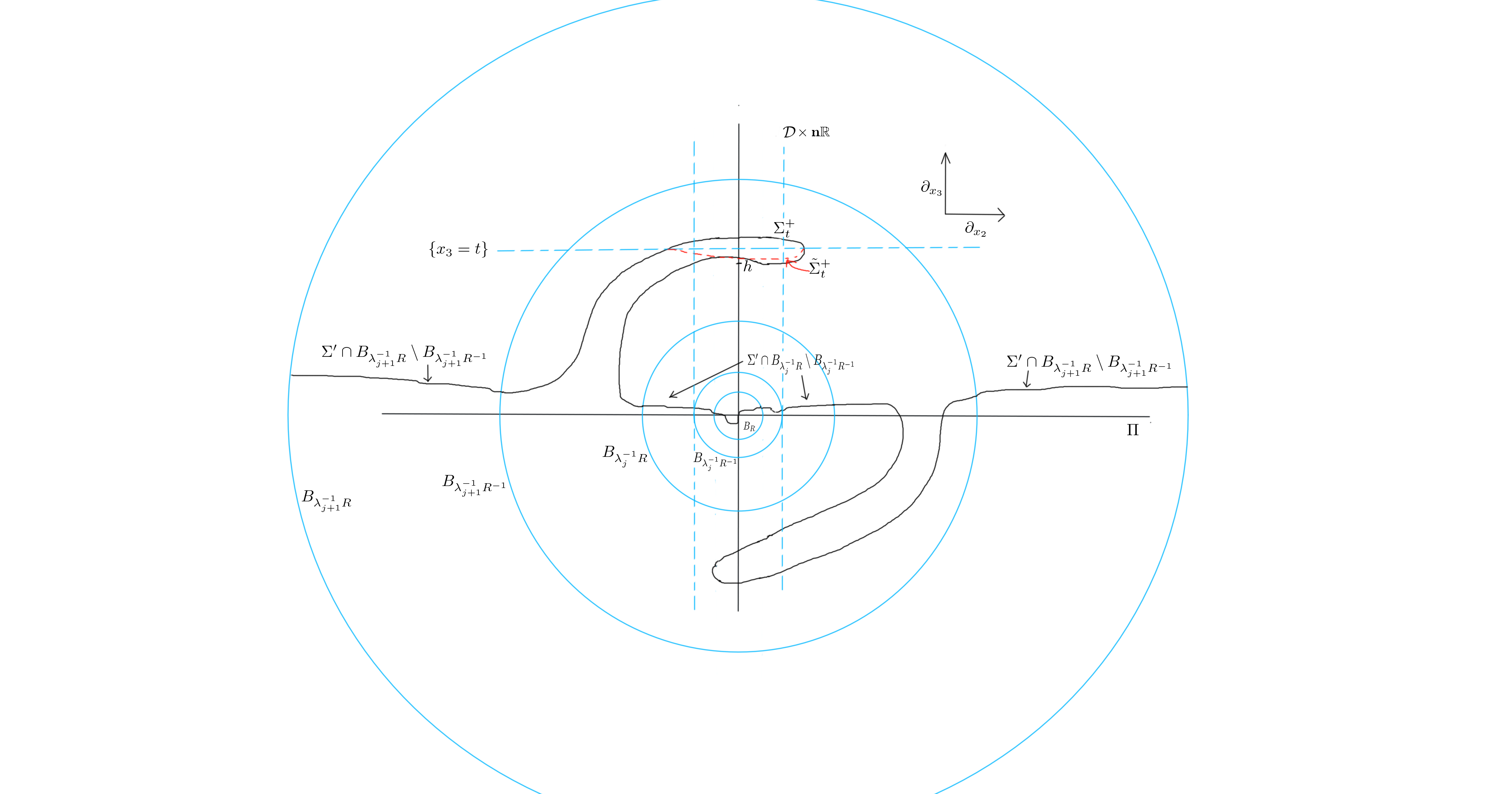}
    \label{fig1}
\end{minipage}

\noindent\begin{minipage}{\textwidth}
    \centering
    \captionof{figure}{Case 2}
    \includegraphics[angle=90, origin=c, scale=0.25, trim={0 0 10cm 0},clip]{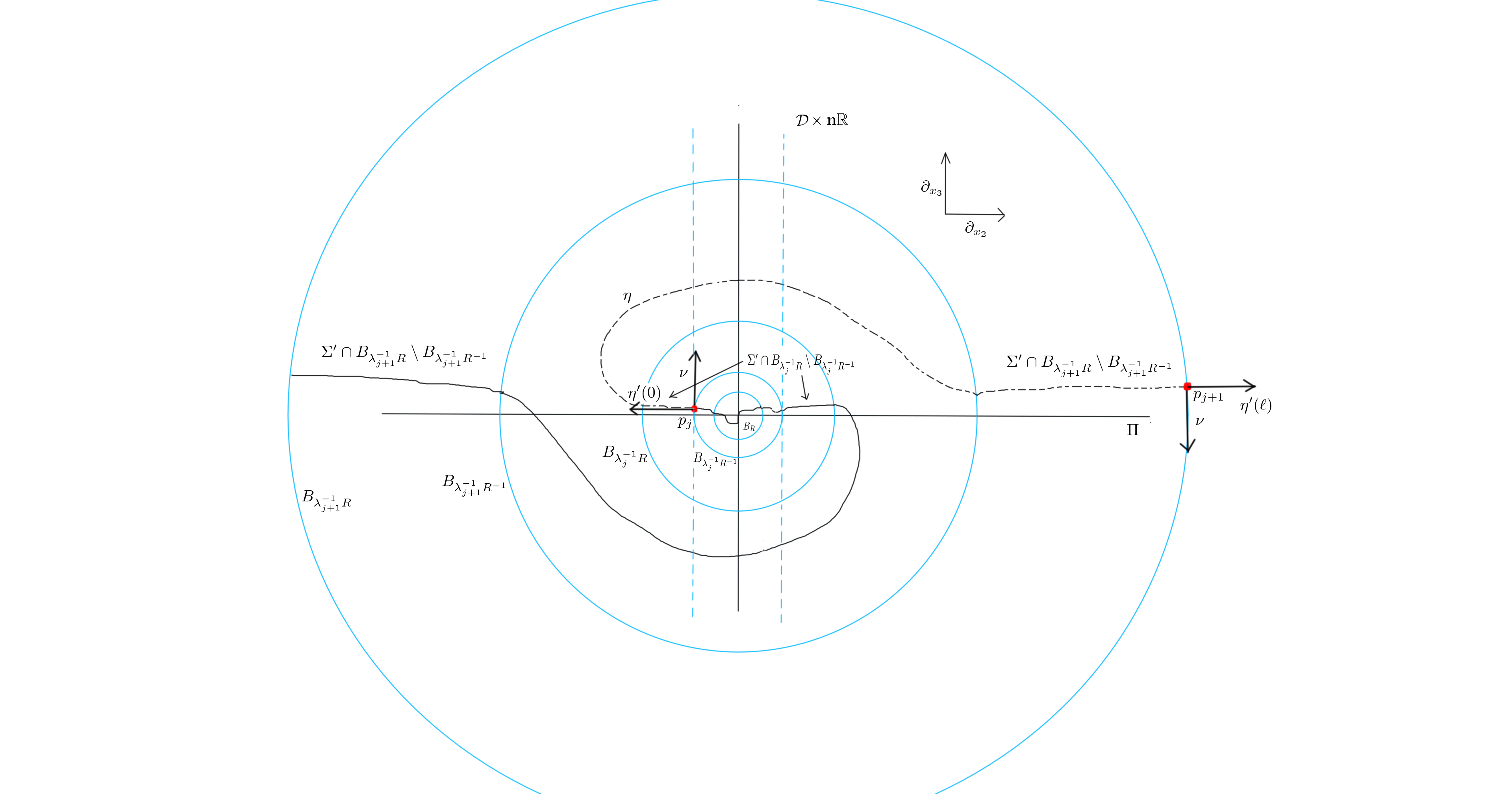}
    \label{fig2}
\end{minipage}

\unappendix

\end{document}